\documentclass[10pt,letterpaper]{amsart}
\usepackage{babel}
\usepackage{amsmath,amsthm,bm}
\usepackage{amsbsy}
\usepackage{amstext}
\usepackage{amsfonts}
\usepackage{floatflt}
\usepackage[pdftex]{graphicx}
\usepackage{mathcomp}
\usepackage{mathtools}
\usepackage[dvipsnames]{xcolor}
\usepackage{tikz,pgfplots}
\usepackage{dsfont}
\usepackage{latexsym}
\usepackage{amssymb}
\usepackage{stmaryrd}
\usepackage{bm}
\usepackage{enumerate}
\usepackage{esint}	
\usepackage[colorlinks,pdfpagelabels,pdfstartview = FitH,bookmarksopen = true,bookmarksnumbered = true,linkcolor = blue,plainpages = false,hypertexnames = false,citecolor = red]{hyperref}
\usepackage{cite}

\allowdisplaybreaks
\makeatletter
\newsavebox{\@brx}
\newcommand{\llangle}[1][]{\savebox{\@brx}{\(\m@th{#1\langle}\)}%
  \mathopen{\copy\@brx\kern-0.5\wd\@brx\usebox{\@brx}}}
\newcommand{\rrangle}[1][]{\savebox{\@brx}{\(\m@th{#1\rangle}\)}%
  \mathclose{\copy\@brx\kern-0.5\wd\@brx\usebox{\@brx}}}
\makeatother

\newtheorem{theorem}{Theorem}
\newtheorem{lemma}[theorem]{Lemma}
\newtheorem{proposition}[theorem]{Proposition}
\newtheorem{corollary}[theorem]{Corollary}

\theoremstyle{definition}
\newtheorem{definition}[theorem]{Definition}

\newtheorem{remark}[theorem]{Remark}

\numberwithin{theorem}{section}
\numberwithin{equation}{section}

\newcommand{\N}{\ensuremath{\mathbb{N}}}
\newcommand{\R}{\ensuremath{\mathbb{R}}}
\newcommand{\mint}{- \mskip-19,5mu \int}

\newcommand{\dx}{\mathrm{d}x}

\makeatletter
\@namedef{subjclassname@2020}{%
  \textup{2020} Mathematics Subject Classification}
\makeatother

\def\Xint#1{\mathchoice
    {\XXint\displaystyle\textstyle{#1}}%
    {\XXint\textstyle\scriptstyle{#1}}%
    {\XXint\scriptstyle\scriptscriptstyle{#1}}%
    {\XXint\scriptscriptstyle\scriptscriptstyle{#1}}%
    \!\int}
\def\XXint#1#2#3{\setbox0=\hbox{$#1{#2#3}{\int}$}
    \vcenter{\hbox{$#2#3$}}\kern-0.5\wd0}
\def\bint{\Xint-}
\def\dashint{\Xint{\raise4pt\hbox to7pt{\hrulefill}}}

\def\XXiint#1#2#3{\setbox0=\hbox{$#1{#2#3}{\iint}$}
    \vcenter{\hbox{$#2#3$}}\kern-0.5\wd0}

\renewcommand{\epsilon}{\varepsilon}

\renewcommand{\rho}{\varrho}

\renewcommand{\epsilon}{\varepsilon}
\renewcommand{\rho}{\varrho}
\renewcommand{\d}{\:\! \mathrm{d}}

\DeclareMathOperator{\loc}{loc}
\DeclareMathOperator{\Tail}{Tail}

\numberwithin{equation}{section}
\allowdisplaybreaks

\subjclass[2020]{}


\begin{document}
\renewcommand{\refname}{References} 
\renewcommand{\abstractname}{Abstract} 
\title[Gradient estimates for fractional $p$-Poisson equation]{Gradient estimates for the \\ fractional $p$-Poisson equation}

\author[V. B\"ogelein]{Verena B\"{o}gelein}
\address{Verena B\"ogelein\\
Fachbereich Mathematik, Universit\"at Salzburg\\
Hellbrunner Str. 34, 5020 Salzburg, Austria}
\email{verena.boegelein@plus.ac.at}

\author[F. Duzaar]{Frank Duzaar}
\address{Frank Duzaar\\
Fachbereich Mathematik, Universit\"at Salzburg\\
Hellbrunner Str. 34, 5020 Salzburg, Austria}
\email{frankjohannes.duzaar@plus.ac.at}

\author[N. Liao]{Naian Liao}
\address{Naian Liao\\
Fachbereich Mathematik, Universit\"at Salzburg\\
Hellbrunner Str. 34, 5020 Salzburg, Austria}
\email{naian.liao@plus.ac.at}

\author[K.~Moring]{Kristian Moring}
\address{Kristian Moring\\
	Fachbereich Mathematik, Universität Salzburg\\
	Hellbrunner Str.~34, 5020 Salzburg, Austria}
\email{kristian.moring@plus.ac.at}

\date{\today}

\begin{abstract}
We consider local weak solutions to the fractional $p$-Poisson equation of order $s$, i.e. $\left( - \Delta_p\right)^s u = f$. In the range $p>1$ and $s\in \big(\frac{p-1}{p},1\big)$ we prove Calder\'on \& Zygmund type estimates at the gradient level. More precisely, we show  for any $q>1$ that
\begin{equation*}
    f\in L^{\frac{qp}{p-1}}_{\rm loc}
    \quad\Longrightarrow\quad
    \nabla u\in L^{qp}_{\rm loc}.
\end{equation*}
The qualitative result is accompanied by a local quantitative estimate.
\end{abstract}
\makeatother

\maketitle

\section{Introduction}
Let $\Omega \subset \R^n$ be a bounded, open set. Consider the $(s,p)$-Poisson equation
\begin{equation} \label{eq:frac-p-lap}
\left( - \Delta_p\right)^s u = f \quad \text{ in } \Omega
\end{equation}
 where $f$ is given in a proper Lebesgue space and the fractional $p$-Laplace operator is defined by
$$
(- \Delta_p )^s u(x) := \mathrm{p.v.} \int_{\R^n} \frac{2 |u(x) - u(y)|^{p-2} (u(x) - u(y) )}{|x-y|^{n+sp}}  \, \d y.
$$
Even though the notion of solutions (see Definition \ref{def:weak-sol}) only involves fractional differentiability of $u$, recent progress has confirmed that the gradient $\nabla u$ actually exists in $L^{p}_{\loc}(\Omega)$ whenever $f\in L^{\frac{p}{p-1}}_{\loc}(\Omega)$ and $s \in \big(\frac{p-1}{p},1\big)$; cf.~\cite{DKLN-higherdiff}.
Our main result goes beyond this regularity in the sense of a  Calder\'on \& Zygmund type estimate.
Notation can be found in \S~\ref{S:notation}. In particular,  $\mathrm{Tail}(\cdot)$ is defined in  \eqref{Eq:tail}. 

\begin{theorem}\label{thm:main}
Let $p\in(1,\infty)$ and $s \in \big(\frac{p-1}{p},1\big)$.
Then, for any local weak solution $u \in W^{s,p}_{\loc}(\Omega) \cap L^{p-1}_{sp}(\R^n)$ to the fractional $p$-Poisson equation~\eqref{eq:frac-p-lap} in the sense of Definition \ref{def:weak-sol} with $f \in L^{\frac{pq}{p-1}}_{\loc}(\Omega)$, $q>1$ we have 
\begin{equation*}
     \nabla u\in L^{pq}_{\loc}\big(\Omega,\R^n\big).
\end{equation*}
Moreover, there exists a constant $c =  c(n,p,s,q)$ such that for any ball $B_{2R} \equiv B_{2R}(x_o) \Subset \Omega$ we have the quantitative local $L^{pq}$-gradient estimate
\begin{align*}
\bigg[ \bint_{B_{R}} |\nabla u|^{pq} \, \d x \bigg]^\frac{1}{q} &\leq c \,\bint_{B_{2R}} |\nabla u|^p \, \d x + \frac{c}{R^{(1-s)p}} \bigg[ \bint_{B_{2R}} |R^s f|^{\frac{qp}{p-1}} \, \d x \bigg]^\frac{1}{q} \\
&\quad  + \frac{c}{R^p} \mathrm{Tail} \big(u-(u)_{B_{2R}};B_{2R}\big)^p.
\end{align*}
\end{theorem}
Although Calder\'on \& Zygmund theory for local and non-local equations has been a very active area of research in recent decades, to our knowledge, Theorem~\ref{thm:main} is the first {\it gradient-level} Calder\'on \& Zygmund type estimate for the fractional $p$-Poisson equation.

\begin{remark}\label{rem:mainth}\upshape
On the right-hand side of the quantitative estimate, the $L^p$-integral of the gradient can be replaced by the fractional Gagliardo semi-norm. In fact, we have
\begin{align*}
    \bigg[ \bint_{B_{R}} &|\nabla u|^{pq} \, \d x \bigg]^\frac{1}{q} \\
    &\le
    \frac{c}{R^{(1-s)p}}
    \mint_{B_{2R}}\int_{B_{2R}}
    \frac{|u(x)-u(y)|^p}{|x-y|^{n+sp}}\,\d x\d y
    +
    \frac{c}{R^{(1-s)p}}\bigg[\mint_{B_{2R}}
    |R^sf|^{\frac{pq}{p-1}}\,\d x\bigg]^\frac1q\\
    &\phantom{\le\,}+
    \frac{c}{R^p}
    \mathrm{Tail}\big(u-(u)_{B_{2R}}; B_{2R}\big)^p
    .
\end{align*}
This does not change the dependence of the constant.
See Corollary~\ref{Cor:Lp} and the remark following it for details.
\hfill $\Box$
\end{remark}

\subsection{Brief background on Calder\'on \& Zygmund estimates for local operators}
The study of integrability for the gradient of weak solutions to elliptic PDEs with an inhomogeneous term has a long tradition. In the linear case, representation of solution is available, and hence the issue can be dealt with by the theory of singular integrals \cite{CZ-1, CZ-2}. However, as one attempts to establish analogous regularity results for nonlinear equations, systems, or even parabolic counterparts, representation formulae are out of order, and moreover, tools from harmonic analysis, though proven profound, seem to become increasingly incompatible. Inevitably, a method that directly involves the equation or system is asked for. In this respect, the spirit to bear in mind is that sharp local regularity estimates for nonlinear equations or systems play an equally effective role as representation formulae for linear ones, cf.~\cite{Iwaniecz, DiBe-Man, Ki-Zh, Caf-Per}.  Still, an important tool from harmonic analysis -- maximal functions -- appeared to be essential in these works. 
A new technique, free of harmonic analysis tools, was developed in \cite{Acerbi-Min}, where Calder\'on \& Zygmund estimates for nonlinear parabolic equations were proved. 
Since then it has witnessed generalizations and applications in many ways. It is worth pointing out that an important point of \cite{Acerbi-Min} in line with \cite{Caf-Per} is that by implementing delicate comparison argument, instead of $C^{1,\alpha}$-estimates, a proper $C^{0,1}$-estimate for the {\it homogeneous} equation suffices to study the integrability of the gradient of solutions to the {\it inhomogeneous} equation. Furthermore, it has been understood that in fact, higher integrability for the gradient of solutions to the {\it homogeneous} equation will actually do a similar job; see~\cite{Kri-Min:06, DuMinSt}. This point of view will be adapted in our work.

\subsection{Gradient regularity in the nonlinear fractional case}
The theory of gradient regularity of weak solutions in the nonlinear, fractional framework has just begun. From a variational point of view, the fractional $p$-Laplace operator can be deemed as a fractional ``cousin" of the standard $p$-Laplace operator, nevertheless the sheer notion of weak solution only involves fractional derivatives, and hence the concept of ``gradient" is not built-in. Recent advances in \cite{BDLMS-1, BDLMS-2, DKLN-higherdiff} have shown that local weak solutions are in fact in higher Sobolev spaces and that the gradient of solutions indeed exists for a certain range of $s$ regardless of the initial low regularity. These works lay the foundation of investigation in the current paper. See also previous important works \cite{Brasco-Lindgren, Brasco-Lindgren-Schikorra} in this direction. 

Prior to our current work, several authors considered the fractional $p$-Poisson equation and examined the impact of the inhomogeneous term on the regularity of solutions; see~\cite{Brasco-Lindgren, Brasco-Lindgren-Schikorra, DKLN-higherdiff, DN-CZ, Lindgren-Garain}. In particular, when $s\in(0,\frac{p-1}{p}]$ the paper \cite{DN-CZ} sharpened some results of \cite{Brasco-Lindgren} and pinpointed the interplay of fractional differentiability and integrability. See also \cite{BK} with a different kind of inhomogeneity and \cite{Bass-Ren, Kuusi-Mingione-Sire-1, Kuusi-Mingione-Sire-2, Schikorra-1} for a self-improving property of Gehring type.

However, none of them dealt with the higher integrability issue at the gradient level. Therefore, our main result represents  a step forward in the direction of quantifying how the higher integrability of $f$ determines the integrability of $\nabla u$. In particular, formally taking $q=1$, Theorem~\ref{thm:main} reduces to a result of \cite{DKLN-higherdiff}. 

Finally, let us remark that notwithstanding the considerable progress made recently, the  theory of gradient regularity for the fractional $p$-Poisson equation, or even simply the homogeneous equation, is still highly malleable. Whenever a more advanced theory is available, we expect it will bear the same wealth of consequences as for the $p$-Laplacian.

\subsection{On the strategy of the proof}

The crux of reaching our main result in Theorem~\ref{thm:main} is a suitable gradient estimate for solutions to~\eqref{eq:frac-p-lap} on super-level sets of the gradient; see~\eqref{est:super-nabla-u}. Since we do not have a Lipschitz type estimate for solutions to the homogeneous equation ($f \equiv 0$) at our disposal, we rely on the $L^q$-bound, for any $q < \infty$, of the gradient of solutions to the homogeneous equation  proved in~\cite{BDLMS-1,BDLMS-2}; see Theorem~\ref{thm:Lq-gradient}.
 The gradient estimate \eqref{est:super-nabla-u} then results from a proper comparison estimate established in Proposition~\ref{lem:comparison-estimate}, which in turn exploits a higher differentiability result from~\cite{DKLN-higherdiff} stated in Theorem~\ref{thm:diening-nondifferentiabledata}. 
Departing from \eqref{est:super-nabla-u}, the desired bound in Theorem~\ref{thm:main} is then deduced via covering and truncation arguments.
In the course of the argument, we also rely on a technical estimate in Lemma~\ref{lem:tail-rho} in order to deal with the so called tail-terms that encode the nonlocal behavior of solutions. We refrain from considering coefficients in the integral kernel, but expect that under the same program the main results remain valid for coefficients of VMO type.
\subsection{Structure of the paper}
We structure the paper as follows. In \S~\ref{S:prelim} we introduce necessary notation used in the paper, notion of solution, algebraic inequalities, and fractional Sobolev spaces. Later in the same section, we also recall some known properties of solutions to \eqref{eq:frac-p-lap} and a tail estimate. In \S~\ref{S:comp-est} we present some comparison estimates which are the key in the proof of the main result. The final argument is then given in \S~\ref{S:proof}.

\section{Preliminaries}\label{S:prelim}
\subsection{Notation}\label{S:notation}
In the entire paper, $p^\prime\equiv\frac{p-1}{p}$ is the H\"older conjugate of $p$, $\N_0\equiv\N\cup\{0\}$, whereas $c$ stands for a generic constant that can change from line to line. In the respective statements and their proofs, we trace the dependencies of the constants on the data and indicate them by writing, for example, $c=c(n,p,s)$ if $c$ depends on $n,p$ and $s$. Next, we denote $B_R(x_o)\Subset \R^n$ as a ball with radius $R$ centered at $x_o\in\R^n$. If $x_o=0$ or if it is clear from the context which center is meant, we omit the center in the notation and write $B_R$ instead of $B_R(x_o)$.

Letting $\Omega\subset\R^n$ be a bounded open set, we introduce the {\bf fractional Sobolev space} $W^{\gamma,q}(\Omega,\R^k)$, with $k\in\N$, $q\in [1,\infty )$ and
$\gamma\in (0,1)$. A measurable function $w\colon\Omega\to\R^k$ will belong to the fractional Sobolev space $W^{\gamma, q}(\Omega,\R^k)$ if and only if
\begin{align*}
    \| w\|_{W^{\gamma ,q}(\Omega,\R^k)}
    &:=
    \| w\|_{L^q(\Omega,\R^k)} +
    [w]_{W^{\gamma ,q}(\Omega,\R^k)}
    <\infty.
\end{align*}
The Gagliardo semi-norm $[w]_{W^{\gamma ,q}(\Omega,\R^k)}$ is defined as 
\begin{align*}
    [w]_{W^{\gamma ,q}(\Omega,\R^k)}
    &:=
    \bigg[
    \iint_{\Omega\times\Omega}
    \frac{|w(x)-w(y)|^q}{|x-y|^{n+\gamma q}}\,\d x\d y
    \bigg]^\frac{1}{q} .
\end{align*}
If the dimension $k$ is clear from the context or if $k=1$, we omit $\R^k$ and write $[w]_{W^{\gamma ,q}(\Omega)}$ instead of $[w]_{W^{\gamma ,q}(\Omega,\R^k)}$.  In particular, if $w=\nabla v$ is the gradient of a scalar function $v\colon\Omega\to \R$, we will write $[\nabla v]_{W^{\gamma ,q}(\Omega)}$ instead of 
$[\nabla v]_{W^{\gamma ,q}(\Omega,\R^n)}$.
Some useful results concerning fractional Sobolev spaces are collected in \S ~\ref{sec:fractional};
for more information we refer to \cite{Hitchhikers-guide}.

\begin{definition}[local weak solution]\label{def:weak-sol}
Let $\Omega\subset\mathbb R^n$ be a bounded open set, $p\in(1,\infty)$, $s\in(0,1)$, and $f\in L^{p'}_{\rm loc}(\Omega)$. A function $u\in W^{s,p}_{\rm loc}(\Omega)$ is called {\bf local weak solution} to the fractional $p$-Poisson equation \eqref{eq:frac-p-lap} in $\Omega$ if and only if
\begin{equation}\label{Lebesgue-weight}
\int_{\mathbb R^n}\frac{|u(x)|^{p-1}}{(1+|x|)^{n+sp}}\, \mathrm{d}x <\infty,
\end{equation}
and
\begin{equation*}
        \iint_{\mathbb R^n\times\mathbb R^n}\frac{|u(x)-u(y)|^{p-2}(u(x)-u(y))(\varphi(x)-\varphi(y))}{|x-y|^{n+sp}}\,\mathrm{d}x \mathrm{d}y =\int_{\Omega} f \varphi \, \d x
\end{equation*}
for every $\varphi\in W^{s,p}(\Omega)$ compactly supported in $\Omega$ and extended to $0$ outside $\Omega$.
\end{definition}

Whenever $u$ satisfies the integrability condition \eqref{Lebesgue-weight}, we will say that $u$ belongs to the weighted Lebesgue space $L^{p-1}_{sp}(\R^n)$. For $x_o\in\Omega$ and $R>0$, with $B_R(x_o)\Subset\Omega$ we introduce the non-local quantity
\begin{equation}\label{Eq:tail}
    {\rm Tail}\big(u; B_R(x_o)\big)
    := 
    \bigg[R^{sp}\int_{\R^n\setminus B_R(x_o)}\frac{|u(x)|^{p-1}}{|x-x_o|^{n+sp}}\,\dx \bigg]^{\frac1{p-1}}.
\end{equation}

\subsection{Algebraic inequalities}
For $\gamma\in(0,\infty)$ and $a\in\R$ we define the \emph{signed $\gamma$-power of $a$} by
$$
    \boldsymbol{a}^\gamma
    :=
    |a|^{\gamma-1} a.
$$
If $a=0$ we set $\boldsymbol{a}^\gamma=0$ also for $\gamma\in(0,1)$. The basic algebraic inequality relating the difference $|\boldsymbol{b}^\gamma - \boldsymbol{a}^\gamma|$ to $|a-b|$ can be found in \cite[Lemma~2.1]{Acerbi-Fusco} for $\gamma\in(0,1)$, and \cite[Lemma~2.2]{GiaquintaModica:1986} for $\gamma\in(1,\infty)$. The stated values of the constants can be derived by a careful inspection of the proofs; 
cf.~\cite[Lemma~2.2]{BDLMS-1}.

\begin{lemma}\label{lem:Acerbi-Fusco}
For any $\gamma>0$, and for all $a,b\in\R$, we have
\begin{align*}
	C_1(|a| + |b|)^{\gamma-1}|b-a|
	\le
	|\boldsymbol{b}^\gamma - \boldsymbol{a}^\gamma|
	\le
	C_2(|a| + |b|)^{\gamma-1}|b-a|,
\end{align*}
where 
\begin{equation*}
    C_1
    =
    \left\{
    \begin{array}{ll}
        \gamma, & \mbox{if $\gamma\in(0,1]$,} \\[5pt]
        2^{1-\gamma}, & \mbox{if $\gamma\in[1,\infty)$,}
    \end{array}
    \right.
    \qquad
    C_2
    =
    \left\{
    \begin{array}{ll}
         2^{1-\gamma}, & \mbox{if $\gamma\in(0,1]$,} \\[5pt]
         \gamma, & \mbox{if $\gamma\in[1,\infty)$.}
    \end{array}
    \right.
\end{equation*}
\end{lemma}

\begin{lemma} \label{lem:elementary-superlevel}
Let $\gamma \in [1,\infty)$, $a,b \in \R^n$ and $|a| \geq K$ for some $K > 0$. Then 
for every $\alpha \in [\gamma,\infty)$
we have
\begin{equation*}
    |a|^\gamma \leq 2^\gamma |a-b|^\gamma + 2^{\alpha -1} K^{\gamma -\alpha} |b|^\alpha.
\end{equation*}
\end{lemma}

\begin{proof} Due to the convexity of $a\mapsto |a|^\gamma$ we have
\begin{equation*} 
    |a|^\gamma\leq 2^{\gamma-1} \big( |a-b|^{\gamma} + |b|^\gamma \big).
\end{equation*}
Thus, in the case $|b| < |a-b|$ the claim is clear. In the reverse case, i.e.~$|a-b| \leq |b|$, we have by assumption
\begin{equation*}
    K \leq |a| \leq |a-b| + |b| \leq 2 |b|,
\end{equation*}
which implies
\begin{equation*}
    |b|^\gamma 
    = 
    |b|^\alpha |b|^{\gamma - \alpha} 
    \leq 
    2^{\alpha - \gamma} 
    K^{\gamma-\alpha} |b|^\alpha.
\end{equation*}
Together with the first case, this concludes the claim.
\end{proof}

\begin{lemma}\label{lem:tech}
Let $0<R_o<R_1<\infty$, $0<\vartheta<1$, $A,B\ge 0$, and $\beta \ge 0 $. Then there exists a constant  $c = c(\beta, \vartheta)$
such that whenever $\phi\colon [R_o,R_1]\to \R$  isa  non-negative bounded function satisfying
\begin{equation*}
	\phi(r_1)
	\le
	\vartheta \phi(r_2) 
    + 
    \frac{A}{(r_2-r_1)^\beta}
    + B
\end{equation*}
for all $R_o\le r_1 <r_2\le R_1$.
Then,  for any $R_o\le \varrho<r\le R_1$ we have
\begin{equation*}
	\phi(\varrho)
	\le
	c\,  
    \bigg[
    \frac{A}{(r-\varrho)^\beta} + B\bigg].
\end{equation*}
\end{lemma}

\subsection{Fractional Sobolev spaces}\label{sec:fractional}
In the following, we summarize the statements about fractional Sobolev spaces that we will use throughout the paper.  We restrict ourselves to the most relevant functional estimates. 
First, we state a Gagliardo-Nirenberg type inequality, see~\cite[Theorem 1]{BM}. Let $s,\,s_1,\,s_2 \geq 0$, $\theta \in (0,1)$ and $1 \leq p,\,p_1,\,p_2 \leq \infty$. The following condition plays a central role in the statement:
\begin{equation} \label{eq:interpolation-cond}
s_2 \in \N_{\geq 1},\quad p_2 =1, \quad \text{ and } \quad s_2-s_1 \leq 1- \tfrac{1}{p_1}.
\end{equation}

\begin{lemma} \label{lem:GN-fract}
Suppose that $\Omega \subset \R^n$ is a bounded Lipschitz domain and 
\begin{equation*}
    \theta \in (0,1), \quad
    s= (1-\theta)s_1 + \theta s_2,\quad
    \tfrac1p = \tfrac{1-\theta}{p_1} + \tfrac{\theta}{p_2}.
\end{equation*}
If~\eqref{eq:interpolation-cond} fails, then there exists a constant $c = c(s_1,s_2,p_1,p_2,\theta,\Omega) >0$ such that for any $u\in 
    W^{s_1,p_1}(\Omega)
    \cap 
    W^{s_2,p_2}(\Omega)$ we have
\begin{equation*}
    \| u \|_{W^{s,p}(\Omega)} 
    \leq 
    c\, \| u \|_{W^{s_1,p_1}(\Omega)}^{1-\theta} 
    \| u \|_{W^{s_2,p_2}(\Omega)}^\theta.
\end{equation*}
\end{lemma}

\begin{corollary}\label{cor:GN}
Let $p > 1$ and $\alpha\in(0,1)$. Then there is a constant $c=c(n,p,\alpha)>0$, so that for each $u \in W^{1+\alpha,p}(B_R)$ we have
\begin{align*}
    \| \nabla u \|_{L^p(B_R)} 
    \leq 
    \frac{c}{R} \| u \|_{L^p(B_R)} +
    c\, \| u \|_{L^p(B_R)}^{\frac{\alpha}{1+\alpha}}
    [\nabla u]_{W^{\alpha,p}(B_R)}^{\frac{1}{1+\alpha}}.
\end{align*} 
\end{corollary}

\begin{proof}
We apply the Gagliardo-Nirenberg type inequality from Lemma~\ref{lem:GN-fract} with $\theta = \frac{1}{1+\alpha}$, $p_1=p_2=p$, $s_1 = 0$, $s_2 = 1+\alpha$, and $s = 1$, and obtain
\begin{align*}
    \| \nabla u \|_{L^p(B_R)} 
    \leq 
    \frac{c}{R} \| u \|_{L^p(B_R)}^\frac{\alpha}{1+\alpha} 
    \Big[\| u \|_{L^p(B_R)} + R\| \nabla u \|_{L^p(B_R)} + R^{1+\alpha}[\nabla u]_{W^{\alpha,p}(B_R)}\Big]^\frac1{1+\alpha}.
\end{align*} 
Note that the dependence on the radius can be derived by a scaling argument, i.e., first deducing the inequality on $B_1$ and then scaling it back to $B_R$. By means of Young's inequality applied with exponents $\frac{1+\alpha}{\alpha}$ and $1+\alpha$, the second term on the right-hand side can be bounded by
$$
    \frac{c}{R^\frac{\alpha}{1+\alpha}} \| u \|_{L^p(B_R)}^\frac{\alpha}{1+\alpha} \| \nabla u \|_{L^p(B_R)}^\frac{1}{1+\alpha}
    \le 
    \tfrac12 \| \nabla u \|_{L^p(B_R)} + \frac{c}{R}\| u \|_{L^p(B_R)},
$$
so that $\tfrac12\| \nabla u \|_{L^p(B_R)}$ can be absorbed on the left-hand side.
\end{proof}
Next, we state two variants of 
Poincar\'e's inequality for functions in fractional Sobolev spaces. 
The first one is standard.
\begin{lemma}\label{lem:poin}
Let $p \geq 1$, $0<s_o \leq s < 1$ and $R > 0$. Then there is a constant $c=c(n,p,s_o)\ge 1$ such that for any 
$u \in W^{s,p}(B_R(x_o))$ we have
\begin{equation*}
    \bint_{B_R(x_o)} |u-(u)_{B_R(x_o)}|^p \, \d x 
    \leq c\, (1-s) R^{sp} \bint_{B_R(x_o)} \int_{B_R(x_o)} \frac{|u(x) - u(y)|^p}{|x-y|^{n+sp}} \, \d x \d y.
\end{equation*}
\end{lemma}

The second one can be readily derived from the first one and applies to functions in $W^{s,p}(B_R)$ that vanish on a set of positive measure in $B_R$, i.e.~have a fat zero level set. It 
is taken from~\cite[Lemma 4.7]{Coz17b}. 

\begin{lemma} \label{lem:coz17b-4.7}
Let $p \geq 1$, $0<s_o \leq s < 1$,
$R > 0$, and $\gamma \in (0,1]$. Then there is a constant $c = c(n,s_o,p,\gamma) \geq 1$ such that
for any $u \in W^{s,p}(B_R)$ that satisfies $u = 0$ a.e.~on a set $\Omega_o \subset B_R$ with $|\Omega_o| \geq \gamma |B_R|$, we have
\begin{equation*}
    \int_{B_R} |u(x)|^p \, \d x 
    \leq 
    c\, (1-s) R^{sp} \iint_{B_R \times B_R} \frac{|u(x)-u(y)|^p}{|x-y|^{n+sp}} \, \d x \d y.
\end{equation*}
\end{lemma}

Finally, we recall the \emph{embedding} $W^{1,p}\hookrightarrow W^{s,p}$.

\begin{lemma} \label{lem:W1q-Wsq-embedding}
Let $p \geq 1$ and $s \in (0,1)$. Then, for any $u \in W^{1,p}(B_R)$ we have 
$$
    \iint_{B_R \times B_R} \frac{|u(x)-u(y)|^p}{|x-y|^{n+sp}} \, \d x \d y 
    \leq 
    c(n) \frac{R^{(1-s)p}}{(1-s)p} \int_{B_R} |\nabla u|^p \, \d x.
$$
\end{lemma}

\subsection{The homogeneous problem}
In this section we summarize the regularity statements of local weak solutions to the homogeneous fractional $p$-Laplace equation 
\begin{equation} \label{eq:frac-p-lap-homo}
    \left( - \Delta_p\right)^s u = 0
    \quad \mbox{in $\Omega$.}
\end{equation}
We start with the $L^\infty$-estimate which is taken from~\cite[Theorem 1.1]{DKP}; see also \cite[Theorem~6.2]{Coz17b}.

\begin{theorem}\label{thm:sup-est}
Let $p \in (1,\infty)$ and $s \in (0,1)$. Then, any local weak solution $u \in W^{s,p}_{\loc}(\Omega) \cap L^{p-1}_{sp}(\R^n)$ to the homogeneous problem~\eqref{eq:frac-p-lap-homo} is locally bounded in $\Omega$, i.e.~$u\in L^\infty_{\loc}(\Omega)$. Moreover, there exists a constant $c=c(n,p,s)$ such that for any ball $B_{2R} = B_{2R}(x_o) \Subset \Omega$ we have 
\begin{align*}
    \|u\|_{L^\infty(B_R)} 
    \le 
    \frac{c}{R^{\frac{n}{p}}} \|u\|_{L^p(B_{2R})} + c\, \Tail\big(u;B_R\big).
\end{align*}
\end{theorem}

Furthermore, the weak gradient exists and is locally in $L^q$ for every $q < \infty$ as has been shown in~\cite[Theorem 1.4]{BDLMS-1} and \cite[Theorem~1.2]{BDLMS-2}.

\begin{theorem} \label{thm:Lq-gradient}
Let $p\in(1,\infty)$ and $s \in (\frac{(p-2)_+}{p},1)$. Then, for any local weak solution $u \in W^{s,p}_{\loc}(\Omega) \cap L^{p-1}_{sp}(\R^n)$ to the homogeneous problem~\eqref{eq:frac-p-lap-homo} we have 
$$
u \in W^{1,q}_{\loc}(\Omega) \quad \text{ for any } q \in [p,\infty).
$$
Moreover, there exists a constant $c =  c(n,p,s,q)$ such that for any ball $B_{2R} \equiv B_{2R}(x_o) \Subset \Omega$ we have
$$
    \| \nabla u \|_{L^q(B_{R})} 
    \leq 
    c\, R^{\frac{n}{q}-1} \left[ R^{1- \frac{n}{p}} 
    \|\nabla u\|_{L^{p}(B_{2R})} + \mathrm{Tail} \big(u-(u)_{B_{2R}};B_{2R}\big) \right].
$$
\end{theorem}

\begin{proof}
From Theorem~\ref{thm:sup-est} we know that $u$ is locally bounded. Next, we note that $w:=u-(u)_{B_{2R}}$ is also a locally bounded, local weak solution to~\eqref{eq:frac-p-lap-homo}.
This allows us to apply \cite[Theorem 1.4]{BDLMS-1} when $p\ge 2$, respectively \cite[Theorem~1.2]{BDLMS-2} when $p\in(1,2)$ to infer that $u \in W^{1,q}_{\loc}(\Omega)$ for any $q \in [p,\infty)$.
 Moreover, there exists a constant $c =  c(n,p,s,q)$ such that for any ball $B_{2R} \equiv B_{2R}(x_o) \Subset \Omega$ we have
\begin{align*}
    \| \nabla u \|_{L^q(B_{R})} 
    &\leq 
    c\, R^{\frac{n}{q}-1} \Big[ R^{s- \frac{n}{p}} [u]_{W^{s,p}(B_{\frac32 R})} + \|w\|_{L^\infty(B_{\frac32 R})} + 
    \mathrm{Tail} \big(w;B_{\frac32R}\big) \Big]. 
\end{align*}
The first term on the right-hand side can be estimated by
Lemma~\ref{lem:W1q-Wsq-embedding}, while the $L^\infty$-norm of $w=u-(u)_{B_{2R}}$  can be bounded by Theorem~\ref{thm:sup-est}.
The latter leads to
\begin{align*}
    \|u-(u)_{B_{2R}}\|_{L^\infty(B_{\frac32 R})}
    &\le
     \frac{c}{R^{\frac{n}{p}}} \|u-(u)_{B_{2R}}\|_{L^p(B_{2R})} + c\Tail\big(u-(u)_{B_{2R}};B_{\frac32R}\big)\\
     &\le
     c\, R^{1 - \frac{n}{p}} \|\nabla u\|_{L^{p}(B_{2R})} 
     +
     c
    \Tail\big(u-(u)_{B_{2R}};B_{\frac32 R}\big).
\end{align*}
To obtain the second line we used Poincar\'e's inequality. Therefore it remains
to control the tail term. This can be achieved by first decomposing the domain of integration into $\R^n\setminus B_{2R}$ and
$B_{2R}\setminus B_{\frac32 R}$.
On the latter we have $|x-x_o|\ge \frac32 R$, which allows to bound the kernel from above.
This leads to
\begin{align*}
    \mathrm{Tail} \big(w;B_{\frac32R}\big)^{p-1}
    &=
    \Big(\frac{3}{4}\Big)^{sp}\mathrm{Tail} \big(w;B_{2R}\big)^{p-1} +
    \Big(\frac{3R}{2}\Big)^{sp} \int_{B_{2R}\setminus B_{\frac{3}{2}R}} \frac{|w|^{p-1}}{|x-x_o|^{n+sp}} \,\d x \\
    &\le 
    \mathrm{Tail} \big(w;B_{2R}\big)^{p-1} +
    \frac{1}{R^n} \int_{B_{2R}} |w|^{p-1} \,\d x.
\end{align*}
Subsequently, we apply Hölder's and Poincar\'e's inequalities to the second integral on the right-hand side.  Combining the various inequalities provides the claimed local $L^q$-gradient estimate. 
\end{proof}

\subsection{The inhomogeneous problem}

We recall the higher differentiability result for weak solutions to the fractional $p$-Poisson equation~\eqref{eq:frac-p-lap} from~\cite[Theorem 1.2]{DKLN-higherdiff}. Here, we formulate the statement tailored to our needs. Note that our $\alpha_o$ defined in \eqref{def:alpha-o} is actually $\alpha_o-1$ from \cite{DKLN-higherdiff}.

\begin{theorem} \label{thm:diening-nondifferentiabledata}
Let $p\in(1,\infty)$, $s \in (0,1)$ with $sp'>1$, and
\begin{equation}\label{def:alpha-o}
    \alpha_o:=
    (sp'-1)\min\{1,p-1\}\in(0,1).
\end{equation}
Then, for any local weak solution $u \in W^{s,p}_{\loc}(\Omega) \cap L^{p-1}_{sp}(\R^n)$ to the fractional $p$-Poisson equation~\eqref{eq:frac-p-lap} in the sense of Definition \ref{def:weak-sol} with $f \in L^{p'}_{\loc}(\Omega)$ we have 
\begin{equation*}
     u\in W^{1+\alpha,p}_{\loc}(\Omega)
     \quad\mbox{for any $\alpha\in(0,\alpha_o)$.}
\end{equation*}
Moreover, there exists a constant $c =  c(n,p,s,\alpha)$ such that 
\begin{align*}
    R^{\alpha+1 - \frac{n}{p}} [\nabla u]_{W^{\alpha, p}(B_R)}
    &\leq 
    c \bigg[ 
    \bint_{B_{2R}} \big|
    u - (u)_{B_{2R}}\big|^p \, \d x \bigg]^\frac{1}{p}\\
    &\phantom{\le\,}+ 
    c\,\mathrm{Tail} \big(u-(u)_{B_{2R}}; B_{2R}\big)  +
    c\, R^{\frac{sp}{p-1} - \frac{n}{p}} \|f\|_{L^{p'}(B_{2R})}^\frac{1}{p-1}
\end{align*}
for any ball $B_{2R} \equiv B_{2R}(x_o) \Subset \Omega$. 
If $u \in W^{s,p}(B_{2R})$ will hold, then $B_{2R} \subset \Omega$ is sufficient.
\end{theorem}

 A quantitative estimate of the $L^p$-norm of the gradient is not explicitly provided in \cite{DKLN-higherdiff}. However, this can easily be derived from the estimates of the second finite differences given there.

\begin{corollary}\label{Cor:Lp}
Under the assumptions of Theorem~\ref{thm:diening-nondifferentiabledata} we have with a constant $c=c(n,p,s)$ for any ball $B_{2R} \equiv B_{2R}(x_o) \Subset \Omega$ the quantitative local $L^p$-gradient estimate 
\begin{align*}
    \|\nabla u\|_{L^p(B_R)}
    \le 
    \frac{c}{R}\Big[ R^{s}[u]_{W^{s,p}(B_{2R})} +
    R^\frac{n}{p}\mathrm{Tail}\big(u-(u)_{B_{2R}}; B_{2R}\big) + R^{sp'}\|f\|^{\frac{1}{p-1}}_{L^{p'}(B_{2R})}\Big].
\end{align*}
\end{corollary}

\begin{proof}
By $\boldsymbol\tau_h u(x):=u(x+h)-u(x)$ we denote the finite difference of $u$ in direction $h\in\R^n\setminus\{0\}$. In the super-quadratic case $p\ge 2$ we apply  after re-scaling
\cite[Lemma 4.5]{DKLN-higherdiff} to $w:=u-(u)_{B_{10R}}$ and conclude that 
\begin{align*}
    \int_{B_{\frac12 R}} \big|\boldsymbol \tau_h \boldsymbol \tau_h w\big|^p \,\dx 
    &\le 
    c\Big(\frac{|h|}{R}\Big)^{s p p'} 
    \Big[ R^{sp}[w]^p_{W^{s,p}(B_{5R})} + \|w\|^p_{L^p(B_{5R})}\\
    &\qquad\qquad\qquad + 
    R^{n}\mathrm{Tail}(w; B_{5R})^p + R^{spp'}\|f\|^{p'}_{L^{p'}(B_{5R})}\Big] ,
\end{align*}
for  any $h\in\R^n\setminus\{0\}$ with $|h|\le \frac{1}{1000}R$.
Recalling $sp'>1$ and using \cite[Chapter 5]{Stein} (see  \cite[Lemma 2.17, (2.6)]{BDLMS-1} for the precise statement), it follows 
\begin{align*}
    \int_{B_{\frac12 R}} \big|\boldsymbol\tau_h w\big|^p \,\dx 
    &\le 
    c\Big(\frac{|h|}{R}\Big)^{p} 
    \Big[ R^{sp}[w]^p_{W^{s,p}(B_{5R})} + 
    \|w\|^p_{L^p(B_{5R})} \\
    &\phantom{\le\, 
    c\Big(\frac{|h|}{R}\Big)^{p} 
    \Big[} +
    R^n\mathrm{Tail}(w; B_{5R})^p + 
    R^{spp'}\|f\|^{p'}_{L^{p'}(B_{5R})}\Big]
\end{align*}
for  any $h\in\R^n\setminus\{0\}$ with $|h|\le \frac{1}{2000}R$.

When it comes to the sub-quadratic case $1<p<2$, we use the estimate following \cite[(4.28)]{DKLN-higherdiff} for the second order finite difference $\boldsymbol \tau_h \boldsymbol \tau_h w$. The power of $|h|$ should be $\frac12 (\alpha+1+\alpha_o)$ in our notation. Subsequently, we take $\alpha=1$, define $\widetilde\gamma=\frac1{2-p}(1+\frac{\alpha_o}{2}-sp)$, and use \cite[(4.28)]{DKLN-higherdiff} to estimate $\|w\|^p_{W^{\widetilde\gamma,p}(B_5)}$  on the right-hand side. This procedure yields
\begin{align*}
    \int_{B_{\frac12 R}} &\big|\boldsymbol \tau_h\boldsymbol\tau_h w\big|^p \,\dx \\
    &\le 
    c\Big(\frac{|h|}{R}\Big)^{\gamma p} 
    \Big[ \|w\|^p_{L^p(B_{10R})} + R^n\mathrm{Tail}\big(w; B_{10R}\big)^p + R^{spp'}\|f\|^{p'}_{L^{p'}(B_{10R})}\Big]
\end{align*}
for  any $h\in\R^n\setminus\{0\}$ with $|h|\le \frac{1}{1000}R$, 
where $\gamma =1+\tfrac12 \alpha_o>1$ and $\alpha_o$ is defined in \eqref{def:alpha-o}. 
Then, \cite[Lemma 2.17, (2.6)]{BDLMS-1} allows one to conclude 
\begin{align*}
    \int_{B_{\frac12R}} &\big|\boldsymbol \tau_h w\big|^p \,\dx \\
   & \le 
    c\Big(\frac{|h|}{R}\Big)^{p} 
    \Big[ \|w\|^p_{L^p(B_{10R})} + R^n\mathrm{Tail}\big(w; B_{10R}\big)^p + R^{spp'}\|f\|^{p'}_{L^{p'}(B_{10R})}\Big]
\end{align*}
for  any $h\in\R^n\setminus\{0\}$ with $|h|\le \frac{1}{2000}R$.
Summarizing the two cases, we end up with a unified estimate for all $p>1$:
\begin{align*}
    \int_{B_{\frac12R}} & \big|\boldsymbol \tau_h w
    \big|^p \,\dx \\ 
    &\le 
    c\Big(\frac{|h|}{R}\Big)^{p} 
    \Big[ R^{sp}[w]^p_{W^{s,p}(B_{10R})} + R^n\mathrm{Tail}\big(w; B_{10R}\big)^p + R^{spp'}\|f\|^{p'}_{L^{p'}(B_{10R})}\Big]
\end{align*}
for  any $h\in\R^n\setminus\{0\}$ with $|h|\le \frac{1}{2000}R$. 
Note that in the case $p\ge2$ we enlarged the domain of integration from $B_{5R}$ to $B_{10R}$ and estimated $\mathrm{Tail}(w; B_{5R})$ in terms  of $\mathrm{Tail}(w; B_{10R})+\|w\|_{L^p(B_{10R})}$. Subsequently, we applied Poincar\'e's inequality from Lemma~\ref{lem:poin} to bound $\|w\|_{L^p(B_{10R})}$ by the Gagliardo semi-norm $[w]_{W^{s,p}(B_{10R})}$.
In turn, a standard estimate for difference quotients ensures 
\begin{align*}
    \int_{B_{\frac12R}} |\nabla w|^p \,\dx 
    \le 
    \frac{c}{R^p}\Big[ R^{sp}[w]^p_{W^{s,p}(B_{10R})} + R^n\mathrm{Tail}\big(w; B_{10R}\big)^p + R^{spp'}\|f\|^{p'}_{L^{p'}(B_{10R})}\Big].
\end{align*}
Since $\nabla w=\nabla u$ and $[w]_{W^{s,p}(B_{10R})}=[u]_{W^{s,p}(B_{10R})}$ we obtain the desired estimate on $B_{\frac12R}$, $B_{10R}$ instead of $B_R$, $B_{2R}$.
Via a standard covering argument we then obtain the claimed gradient estimate on $B_R$, $B_{2R}$. 
\end{proof}

\begin{remark}\upshape
In averaged form and with integrals written out instead of norms 
the inequality of the corollary is as follows
\begin{align*}
    \mint_{B_R}&|\nabla u|^p\,
    \d x\\
    &\le \frac{c}{R^{(1-s)p}}
    \mint_{B_{2R}}\int_{B_{2R}}
    \frac{|u(x)-u(y)|^p}{|x-y|^{n+sp}}\,\d x\d y
    +
    \frac{c}{R^p}
    \mathrm{Tail}\big(u-(u)_{B_{2R}}; B_{2R}\big)^p\\
    &\phantom{\le\,}+
    \frac{c}{R^{(1-s)p}}\mint_{B_{2R}}
    |R^sf|^{p'}\,\d x\\
    &\le
    \frac{c}{R^{(1-s)p}}
    \mint_{B_{2R}}\int_{B_{2R}}
    \frac{|u(x)-u(y)|^p}{|x-y|^{n+sp}}\,\d x\d y
    +
    \frac{c}{R^p}
    \mathrm{Tail}\big(u-(u)_{B_{2R}}; B_{2R}\big)^p\\
    &\phantom{\le\,}+
    \frac{c}{R^{(1-s)p}}\bigg[\mint_{B_{2R}}
    |R^sf|^{qp'}\,\d x\bigg]^\frac1q.
\end{align*}
To obtain the last line, we have used H\"older's inequality. The preceding inequality  allows us to replace the $L^p$-integral of $\nabla u$ by the fractional Gagliardo semi-norm of $u$ in the main theorem as stated in Remark~\ref{rem:mainth}. \hfill $\Box$
\end{remark}

In the proof of the Calder\'on \& Zygmund  estimate we first consider the case $R=1$ and $x_o=0$, and prove the result on $B_1, B_2$. The general case is then obtained using a scaling argument based on the following simple observation.  

\begin{lemma}[Scaling] \label{lem:equation-scaling}
Let $p\in(1,\infty)$ and $s \in (0,1)$. Suppose $f \in L^{p'}(B_R(x_o))$ and $w \in W^{s,p}(B_R(x_o)) \cap L^{p-1}_{sp} (\R^n)$ is a weak solution to the fractional $p$-Poisson equation
\begin{equation*}
    (- \Delta_p)^s w = f\quad \mbox{ in $B_R(x_o)$.}
\end{equation*}
Then, $w_R(x) := \frac{w(x_o+Rx)}{R^s} \in W^{s,p}(B_1(0)) \cap L^{p-1}_{sp} (\R^n)$ is a weak solution to the
$p$-Poisson equation in $B_1(0)$ with inhomogeneity 
$f_R(x) := R^s f(x_o+Rx) \in L^{p'}(B_1(0))$, i.e.~we have
$$
(- \Delta_p)^s w_R = f_R \quad \mbox{in $B_1(0)$.}
$$
\end{lemma}

\subsection{Tail estimate}
In the proof of the Calder\'on \& Zygmund estimate, we use the following tail estimate, which is a generalization of~\cite[Lemma 2.6]{BK}. In case $p = 2$, a tail estimate of similar type is proved in~\cite[Lemma 2.9]{KNS-22}.

\begin{lemma}\label{lem:tail-rho}
Let $p>1$ and $s\in(0,1)$. There exists a constant $c = c(s,n,p)>0$ such that for  
$u \in W^{1,p}(B_2) \cap L^{p-1}_{sp}(\R^n)$ and $x_o\in B_2$, $R\in(0,1]$ with $B_{R}(x_o)\subset B_2$ and $\rho>0$, $k\in\N_0$ with $\frac12 R\le 2^k\rho< R$ we have
\begin{align*} 
    \mathrm{Tail} \big(u&-(u)_{B_{\rho}(x_o)}; B_\rho (x_o)\big) \\
    &\leq 
    \frac{c}{R^{\frac{n}{p-1}}}\Big(\frac{\rho}{R}\Big)^{sp'} 
    \Bigg[\mathrm{Tail} \big(u-(u)_{B_{2}};B_{2}\big) +
    \bigg[
    \bint_{B_{2}}
    |\nabla u|^{p} \,\d x\bigg]^\frac{1}{p} \Bigg]\\\nonumber
    &\quad + 
    c \,\rho \sum_{j=1}^{k} 2^{-j(sp' - 1)} \bigg[ \bint_{B_{2^j \rho}(x_o)} |\nabla u|^p \, \d x \bigg]^\frac{1}{p}.
\end{align*}
If $k=0$, the last term is interpreted as zero. 
\end{lemma}

\begin{proof}
Let us re-write the left-hand side in the form
\begin{equation*}
    \mathrm{Tail} \big(u-(u)_{B_{\rho}(x_o)}; B_\rho (x_o)\big)^{p-1}
    =
    \mathbf T^{p-1} + \sum_{j=0}^{k-1}\mathbf B_{j}^{p-1},
\end{equation*}
where we abbreviated
\begin{equation*}
    \mathbf T
    :=
    \bigg[
    \rho^{sp}
    \int_{\mathbb{R}^{n}\setminus B_{2^k\rho}(x_o)}\frac{|u-(u)_{B_{\rho}(x_o)}|^{p-1}}{|x-x_o|^{n+sp}} \,\d x\bigg]^\frac1{p-1},
\end{equation*}
and
\begin{equation*}
    \mathbf B_{j}
    :=
    \bigg[
    \rho^{sp}
    \int_{B_{2^{j+1}\rho}(x_o)\setminus B_{2^{j}\rho}(x_o)}\frac{|u-(u)_{B_{\rho}(x_o)}|^{p-1}}{|x-x_o|^{n+sp}} \,\d x\bigg]^\frac1{p-1} .
\end{equation*}
If $k=0$, the sum $\sum_{j=0}^{k-1}\mathbf B_{j}^{p-1}$ has to be interpreted as zero. 
To estimate the $\mathbf T$-term, we apply  Minkowski's inequality in the case $p\ge 2$, while for $1<p\le2$ we first use the elementary inequality $(a+b)^{p-1}\le a^{p-1} +b^{p-1}$ and then the convexity of $t\mapsto t^\frac1{p-1}$. 
This yields 
\begin{align*}
    \mathbf T
    &\le 
    c\bigg[\rho^{sp}
    \int_{\mathbb{R}^{n}\setminus B_{2^k\rho}(x_o)}\frac{|u-(u)_{B_2}|^{p-1}}{|x-x_o|^{n+sp}} \,\d x\bigg]^{\frac{1}{p-1}} \\
    &\quad +
    c\bigg[\rho^{sp}
    \int_{\mathbb{R}^{n}\setminus B_{2^k\rho}(x_o)}\frac{|(u)_{B_2}-(u)_{B_{\rho}(x_o)}|^{p-1}}{|x-x_o|^{n+sp}} \,\d x\bigg]^{\frac{1}{p-1}}
    =:
    c\,\mathbf T_1 + c\,\mathbf T_2,
\end{align*}
where the constant  $c$ depends only on $p$. We now estimate $\mathbf T_1$ and $\mathbf T_2$ one by one, starting with $\mathbf T_1$. In the integral of $\mathbf T_1^{p-1}$ we decompose the domain of integration  into $\mathbb{R}^{n}\setminus B_{2}$ and $B_2\setminus B_{2^k\rho}(x_o)$. This yields
\begin{equation*}
    \mathbf T_1^{p-1}
    =
    \rho^{sp}
    \int_{\mathbb{R}^{n}\setminus B_{2}}\frac{|u-(u)_{B_2}|^{p-1}}{|x-x_o|^{n+sp}} \,\d x +
    \rho^{sp}\int_{B_2\setminus B_{2^k\rho}(x_o)}\frac{|u-(u)_{B_2}|^{p-1}}{|x-x_o|^{n+sp}} \,\d x .
\end{equation*}
To estimate the first term, we note that $B_{R}(x_o)\subset B_2$ and $x\in \R^n \setminus B_2$ imply that
\[
\frac{|x-x_o|}{|x|}\ge 1-\frac{|x_o|}{|x|}\ge 1-\frac{|x_o|}{2}\ge 1-\frac{2-R}{2}=\frac{R}2.
\]
For the second term, i.e.~the integral over
$B_2\setminus B_{2^k\rho}(x_o)$, we have 
$|x-x_o|\ge 2^k\rho \ge \tfrac12 R$ as bound from below. Together with Hölder's and Poincar\'e's inequalities this yields
\begin{align*}
    \mathbf T_1^{p-1}
    &\le 
    \frac{2^{n+sp}}{R^n}\Big(\frac{\rho}{R}\Big)^{sp} \bigg[
    \int_{\mathbb{R}^{n}\setminus B_{2}}\frac{|u-(u)_{B_2}|^{p-1}}{|x|^{n+sp}} \,\d x +
    \int_{B_2}|u-(u)_{B_2}|^{p-1} \,\d x\bigg] \\
    &\le 
    \frac{c}{R^n}\Big(\frac{\rho}{R}\Big)^{sp} \Bigg[
    \mathrm{Tail} \big(u-(u)_{B_{2}};B_{2}\big)^{p-1} +
    \bigg[\bint_{B_2}|u-(u)_{B_2}|^{p} \,\d x\bigg]^{\frac{p-1}{p}}\Bigg] \\
    &\le 
    \frac{c}{R^n}\Big(\frac{\rho}{R}\Big)^{sp} \Bigg[
    \mathrm{Tail} \big(u-(u)_{B_{2}};B_{2}\big) +
    \bigg[\bint_{B_2}|\nabla u|^{p} \,\d x\bigg]^{\frac{1}{p}}\Bigg]^{p-1},
\end{align*}
with $c=c(n,p)$.
To estimate the  term $\mathbf T_2$ we use 
\begin{equation*}
    \sigma^{sp}
    \int_{\mathbb{R}^{n}\setminus B_{\sigma}(x_o)}\frac{1}{|x-x_o|^{n+sp}} 
    \,\d x
    =
    \frac{n|B_1|}{sp},
\end{equation*}
and $2^{-k}\le 2\rho/R$. This provides us with  
\begin{align*}
    \mathbf T_2
    &\le 
    c\, 2^{-ksp'}
    \big|(u)_{B_2}-(u)_{B_{\rho}(x_o)}\big| \\
    &\le 
    c \Big(\frac{\rho}{R}\Big)^{sp'}
    \big|(u)_{B_{2}}-(u)_{B_{2^k\rho}(x_o)}\big| +
    c\,2^{-ksp'} \big|(u)_{B_{2^k\rho}(x_o)}-(u)_{B_{\rho}(x_o)}\big| ,
\end{align*}
where $c=c(n,p,s)$.
Subsequently, the first term on the right-hand side is further estimated by Hölder's and Poincar\'e's inequalities as 
\begin{align*}
    \big|(u)_{B_{2}}-(u)_{B_{2^k\rho}(x_o)}\big| 
    &\le 
    c\bigg[\bint_{B_{2^{k}\rho}(x_o)}
    |u-(u)_{B_2}|^{p} \,\d x\bigg]^\frac{1}{p} \\
    &\le 
    \frac{c}{R^{\frac{n}{p}}}  
    \bigg[\bint_{B_{2}}
    |u-(u)_{B_{2}}|^{p} \,\d x\bigg]^\frac{1}{p} \le 
    \frac{c}{R^{\frac{n}{p-1}}}  
    \bigg[
    \bint_{B_{2}}
    |\nabla u|^{p} \,\d x\bigg]^\frac{1}{p}.
\end{align*}
To obtain the second-to-last line we enlarged the domain of integration and used the lower bound for $2^k\rho\ge \frac12 R$, whereas for the last line we used $R\le 1$ in addition to  Poincar\'e's inequality. The second term on the right-hand side in the estimate of $\mathbf T_2$ is treated in a similar way. In fact, we have
\begin{align*}
    2^{-ksp'}\big|(u)_{B_{2^k\rho}(x_o)}-(u)_{B_{\rho}(x_o)}\big| 
    &\le 
    2^{-ksp'}\sum_{j=0}^{k-1} \big|(u)_{B_{2^{j+1}\rho}(x_o)}-(u)_{B_{2^j\rho}(x_o)}\big| \\
    &\le 
    c\,2^{-ksp'}\sum_{j=0}^{k-1} \bigg[\bint_{B_{2^{j+1}\rho}(x_o)}
    \big|
    u-(u)_{B_{2^{j+1}\rho}(x_o)}
    \big|^{p} \,\d x\bigg]^\frac{1}{p} \\
    &\le 
    c\,2^{-ksp'}\rho \sum_{j=1}^{k} 
    2^{j}  
    \bigg[
    \bint_{B_{2^{j}\rho}(x_o)}
    |\nabla u|^{p} \,\d x\bigg]^\frac{1}{p} \\
    &\le 
    c\,\rho \sum_{j=1}^{k} 
    2^{-j(sp'-1)}  
    \bigg[
    \bint_{B_{2^{j}\rho}(x_o)}
    |\nabla u|^{p} \,\d x\bigg]^\frac{1}{p} .
\end{align*}
The $j$-th term $\mathbf B_j$ is treated in the same way as the $\mathbf T$-term. Namely, we first estimate the denominator $|x-x_o|$ from below by $2^j\rho$. Then,  we distinguish between the cases $p\ge 2$ and $1<p<2$ to split the integral and finally apply H\"older's inequality. This way we get
\begin{align*}
    \mathbf B_{j}
    &\leq 
    c\bigg[
    2^{-jsp}
    \bint_{B_{2^{j+1}\rho}}
    \big|
    u-(u)_{B_{\rho}}
    \big|^{p-1} \,\d x\bigg]^\frac{1}{p-1}\\
    &\leq 
    c\,2^{-jsp'}
    \bigg[
    \bigg(\bint_{B_{2^{j+1}\rho}}
    \big|
    u-(u)_{B_{2^{j+1}\rho}}
    \big|^{p-1} \,\d x\bigg)^\frac{1}{p-1} 
    +
    \big|(u)_{B_{2^{j+1}\rho}}-(u)_{B_{\rho}}\big|\bigg]\\
    &\leq 
    c\,2^{-jsp'}
    \bigg[\bigg(
    \bint_{B_{2^{j+1}\rho}}
    \big|
    u-(u)_{B_{2^{j+1}\rho}}
    \big|^{p} \,\d x\bigg)^\frac{1}{p}+\sum_{i=0}^{j}\big|(u)_{B_{2^{i+1}\rho}}-(u)_{B_{2^{i}\rho}}\big|\bigg]\\
    &\leq 
    c\,2^{-jsp'}
    \sum_{i=1}^{j+1}
    \bigg[\bint_{B_{2^{i}\rho}}
    \big|u-(u)_{B_{2^{i}\rho}}
    \big|^{p} \,\d x\bigg]^\frac{1}{p} .
\end{align*}
Summing over $j$ from $0$ to $k-1$, applying Poincar\'e's inequality and interchanging the order of the summations then leads to
\begin{align*}
    \sum_{j=0}^{k-1}
    \mathbf B_{j}
    &\le 
    c\,\rho \sum_{j=0}^{k-1} 
    2^{-jsp'}\sum_{i=1}^{j+1} 
    2^i \bigg[\bint_{B_{2^{i}\rho}}|\nabla u|^{p}\d x\bigg]^\frac{1}{p} \\
    &=
    c\,\rho 
    \sum_{i=1}^{k}
    2^i \bigg[\bint_{B_{2^{i}\rho}}|\nabla u|^{p}\d x\bigg]^\frac{1}{p}
    \sum_{j=i-1}^{k-1} 2^{-jsp'}\\
    &\le
    c\,\rho \sum_{i=1}^{k} 2^{-i(sp'-1)} \bigg[\bint_{B_{2^{i}\rho}}|\nabla u|^{p}\d x\bigg]^\frac{1}{p}.
\end{align*}
To obtain the last line we used
\begin{align*}
    \sum_{j=i-1}^{k-1} 2^{-jsp'}
    &=
    2^{-(i-1)sp'} \sum_{j=0}^{k-i} 2^{-jsp'}
    \le 
    \frac{2^{-(i-1)sp'}}{1-2^{-sp'}}.
\end{align*}
Collecting these estimates gives the desired conclusion.
\end{proof}

\section{Comparison estimates}\label{S:comp-est} 
Our goal in this section is to prove comparison estimates for the difference between the original solution and the solution to an associated homogeneous Dirchlet problem. The precise setup is as follows. We consider a ball $B_{R}(x_o)$ such that $B_{2R}(x_o) \Subset \Omega$ and denote by $v\in W^{s,p}(B_R(x_o)) \cap L^{p-1}_{sp} (\R^n)$ the unique weak solution to the homogeneous Dirichlet problem 
\begin{equation}\label{CD-homo} 
\left\{
    \begin{array}{cl}
      (- \Delta_p)^s v  = 0 & \mbox{in $B_R(x_o)$,} \\[6pt]
      v = u & \mbox{a.e.~in $\R^n \setminus B_R(x_o)$.}
    \end{array}
\right.
\end{equation}
Before proceeding to the comparison estimates, we summarize in the following remark the available qualitative properties of the gradient of $u$ and $v$, which will be called upon later at our will. 

\begin{remark}\label{rem:W^1p}\upshape
Let $p>1$, $s\in (0,1)$ with $sp'>1$, $f\in L^{p'}_{\loc}(\Omega)$, and $B_{2R}(x_o)\Subset\Omega$.
Under these assumptions, Theorem~\ref{thm:diening-nondifferentiabledata} guarantees that on one hand $u\in W^{1+\alpha, p}_{\loc}(\Omega)$ for each $\alpha\in (0,\alpha_o)$, where $\alpha_o\in(0,1)$ is defined in \eqref{def:alpha-o}, while on the other hand $v\in W^{1+\alpha, p}_{\loc}(B_R(x_o))$. In particular, this means that $\nabla u\in L^p_{\rm loc}(\Omega)$ and $\nabla v\in L^p_{\rm loc}(B_R(x_o))$.
\end{remark}

As a preliminary result, we first derive a comparison estimate between $u$ and $v$ in terms of the $W^{s,p}$-norm. Note that it does not involve the gradient of solutions.

\begin{lemma} \label{lem:u-v-fractest}
Let $p>1$, $s\in(0,1)$, $f \in L^{p'}_{\loc}(\Omega)$, $B_{2R}(x_o) \Subset \Omega$ and let $u$ be a weak solution to~\eqref{eq:frac-p-lap} in the sense Definition~\ref{def:weak-sol} and $v$ the unique weak solution to the homogeneous Dirichlet problem \eqref{CD-homo} and denote $w := u-v$. There exists a constant $c=c(n,p,s)$, such that
if $p \in[2,\infty)$, then we have
\begin{align} \label{eq:u-v-fractest}
    \iint_{\R^n \times  \R^n} \frac{|w(x) - w(y)|^p}{|x-y|^{n+sp}} \, \d x \d y 
    \leq 
    c\, R^{sp'} \int_{B_R}  |f|^{p'} \, \d x,
\end{align}
whereas if $p\in(1,2]$, then for any $\delta\in(0,1]$ we have
\begin{align}\label{eq:u-v-fractest-p<2}\nonumber
    \iint_{B_{2R} \times B_{2R}}& \frac{|w(x) - w(y)|^p}{|x-y|^{n+sp}} \, \d x \d y \\
    &\leq 
    \delta \iint_{B_{2R} \times B_{2R}} \frac{|u(x) - u(y)|^p}{|x-y|^{n+sp}} \, \d x \d y +
    \frac{c}{\delta^\frac{2-p}{p-1}} R^{sp'} \int_{B_R}  |f|^{p'} \, \d x.
\end{align}
\end{lemma}

\begin{proof}
Taking the difference of the equations for $u$ and $v$ and testing the result with $w$, we obtain
\begin{align}\label{comp-I}
    \mathbf{I}
    &:=
    \iint_{\R^n \times \R^n} 
    \frac{\big(\boldsymbol{(u(x) - u(y))}^{p-1} - \boldsymbol{(v(x) - v(y))}^{p-1}\big) (w(x) - w(y))}{|x-y|^{n+sp}} \, \d x \d y \nonumber\\
    &\, = 
    \int_{B_R} f w \, \d x .
\end{align}
Noting that 
\begin{equation}\label{wx-y}
    w(x) - w(y)
    =
    (u(x)-u(y)) - (v(x)-v(y)).
\end{equation}
the integrand of the integral on the left-hand side in~\eqref{comp-I} is non-negative due to monotonicity of the map $\R\ni a\mapsto \boldsymbol{a}^{p-1}$. By using H\"older's inequality, Lemma~\ref{lem:coz17b-4.7} (note that $w\equiv 0$ on $B_{2R}\setminus B_R$) and Young's inequality we obtain for the right-hand side
of \eqref{comp-I} for any $\epsilon\in(0,1)$ that
\begin{align}\label{est:fw}\nonumber
    \int_{B_R} f w \, \d x 
    &\leq 
    \|f\|_{L^{p'}(B_R)} \|w\|_{L^{p}(B_{R})} 
    =
    \|f\|_{L^{p'}(B_R)} \|w\|_{L^{p}(B_{2R})}\\\nonumber
    &\leq 
    c\, R^s \|f\|_{L^{p'}(B_R)} [w]_{W^{s,p}(B_{2R})} \\
    &\leq 
    \epsilon  [w]_{W^{s,p}(B_{2R})}^p + 
    c\,\epsilon^{-\frac{1}{p-1}} R^{sp'} \|f\|_{L^{p'}(B_R)}^{p'},
\end{align}
where $c=c(n,p,s)$. In the super-quadratic case $p\ge 2$, we deduce from Lemma~\ref{lem:Acerbi-Fusco} that 
\begin{equation*}
    \big(\boldsymbol{b}^{p-1}- \boldsymbol{a}^{p-1}\big)(b-a) \geq \tfrac{1}{c(p)} |b-a|^p
    \quad\mbox{for any $a,b\in \R$,}
\end{equation*}
so that
\begin{align*}
    \iint_{\R^n \times \R^n} \frac{|w(x) - w(y)|^p}{|x-y|^{n+sp}} \, \d x \d y  
    &\leq 
    c(p) \int_{B_R} f w \, \d x \\
    &\le c(p)
    \Big[
    \epsilon  [w]_{W^{s,p}(B_{2R})}^p + 
    c\epsilon^{-\frac{1}{p-1}} R^{sp'} \|f\|_{L^{p'}(B_R)}^{p'}
    \Big]\\
    &\leq 
    \tfrac12  [w]_{W^{s,p}(B_{2R})}^p + c\, R^{sp'} \|f\|_{L^{p'}(B_R)}^{p'} .
\end{align*}
To obtain the second-to-last line, we  used \eqref{est:fw}, whereas in the last line we fixed $\epsilon=\frac1{2c(p)}$.
We absorb the first term of the right-hand side on the left to conclude the proof of this case.

In the sub-quadratic case $p\in(1,2)$, a lower bound for the left-hand side integral in \eqref{comp-I}~results from~\eqref{wx-y} and Lemma~\ref{lem:Acerbi-Fusco}. In fact, we have
\begin{align*}
    \mathfrak D
    &:=\big(\boldsymbol{(u(x)-u(y))}^{p-1}-
    \boldsymbol{(v(x)-v(y))}^{p-1}\big)
    (w(x)-w(y)) \\
    &\ge 
    \tfrac{1}{c(p)}
    \big(|u(x)-u(y)|+|v(x)-v(y)|\big)^{p-2}|w(x)-w(y)|^2.
\end{align*}
Applying first Young's inequality with exponents $(\frac{2}{2-p},\frac2p)$ and then the last display, we obtain
\begin{align*}
    |w(x)-w(y)|^p
    &=
    \delta^{\frac{2-p}{2}}\big(|u(x)-u(y)|+|v(x)-v(y)|\big)^{\frac{(2-p)p}{2}}  \\
    &\phantom{\le\,}
    \cdot
    \delta^{-\frac{2-p}{2}} \big(|u(x)-u(y)|+|v(x)-v(y)|\big)^{\frac{(p-2)p}{2}}|w(x)-w(y)|^p\\
    &\le 
    \delta \big(|u(x)-u(y)|+|v(x)-v(y)|\big)^p\\
    &\phantom{\le\,} +
    \delta^{-\frac{2-p}{p}} \big(|u(x)-u(y)|+|v(x)-v(y)|\big)^{p-2}|w(x)-w(y)|^2\\
    &\le 
    8\delta |u(x)-u(y)|^p + 2\delta |w(x)-w(y)|^p + 
    \frac{c(p)}{\delta^{\frac{2-p}{p}}}\mathfrak D.
\end{align*}
If the value of $\delta$ is restricted to $(0,\frac14]$ in the above algebraic inequality, the second term on the right-hand side can be absorbed on the left, so that 
\begin{align*}
    |w(x)-w(y)|^p
    &\le 
    16\delta |u(x)-u(y)|^p  + 
    \frac{c(p)}{\delta^{\frac{2-p}{p}}}\mathfrak D.
\end{align*}
Dividing both sides by $|x-y|^{n+sp}$, then integrating both $x$ and $y$ over $B_{2R}$ and finally recalling that the integrand of $\mathbf{I}$ in \eqref{comp-I} is non-negative, we obtain 
\begin{align*}
    [w]_{W^{s,p}(B_{2R})}^p 
    &\leq 
    16\delta [u]_{W^{s,p}(B_{2R})}^p +
    \frac{c(p)}{\delta^{\frac{2-p}{p}}} \mathbf{I} \\
    &\le 
    16\delta [u]_{W^{s,p}(B_{2R})}^p +
    \frac{c}{\delta^{\frac{2-p}{p}}}  
    \Big[\epsilon[w]_{W^{s,p}(B_{2R})}^p + 
    c\,\epsilon^{-\frac{1}{p-1}} R^{sp'} \|f\|_{L^{p'}(B_R)}^{p'}\Big],
\end{align*}
where $c=c(N,p,s)$.
Choosing $\epsilon=\frac{1}{2c}\delta^{\frac{2-p}{p}}$ allows to re-absorb $[w]_{W^{s,p}(B_{2R})}^p$ on the left-hand side. Replacing $32 \delta$ that will appear in the front of the first term on the right-hand side of the resulting inequality  by $\delta$, i.e.~the one after re-absorbing $[w]_{W^{s,p}(B_{2R})}^p,$   yields the estimate claimed in this case as well.
\end{proof}

The next corollary unifies the cases $p\le2$ and $p\ge2$. Together with the Poincar\'e type inequality from Lemma~\ref{lem:coz17b-4.7}, this leads to a comparison estimate for $u-v$ in terms of the $L^p$-integral of $\nabla u$.

\begin{corollary}[zero-order comparison estimate]\label{cor:u-v-fractest}
Let $p>1$, $s\in(0,1)$ with $sp'>1$, $f \in L^{p'}_{\loc}(\Omega)$, and $B_{2R}(x_o) \Subset \Omega$.  Furthermore, let $u$ be a weak solution to~\eqref{eq:frac-p-lap} in the sense of  Definition~\ref{def:weak-sol} and $v$ the unique weak solution to the homogeneous Dirichlet problem \eqref{CD-homo}. Then, there exists a constant $c=c(n,p,s)$, such that
\begin{align*}
    \int_{B_{R}} |u-v|^p \, \d x 
    \leq
    \delta R^p\int_{B_{2R}} |\nabla u|^p \, \d x +
    \frac{c}{\delta^\frac1{p-1}} R^{spp'} \int_{B_R}  |f|^{p'} \, \d x
\end{align*}
for any $\delta\in(0,1]$.
\end{corollary}

\begin{proof}
Joining \eqref{eq:u-v-fractest} and 
\eqref{eq:u-v-fractest-p<2} from Lemma~\ref{lem:u-v-fractest} (note that $\frac{2-p}{p-1}\le \frac1{p-1}$ for $p\in(1,2]$) and subsequently applying Lemma~\ref{lem:W1q-Wsq-embedding} to $[u]_{W^{s,p}(B_{2R})}^p$, we have
\begin{align*}
    [u-v]_{W^{s,p}(B_{2R})}^p
    &\le
    \delta [u]_{W^{s,p}(B_{2R})}^p  +
    \frac{c}{\delta^\frac{1}{p-1}} R^{sp'} \|f\|_{L^{p'}(B_R)}^{p'}\\
    &\le
    c\,\delta R^{(1-s)p}\|\nabla u\|_{L^{p}(B_{2R})}^p +
    \frac{c\,R^{sp'}}{\delta^\frac1{p-1}}  \|f\|_{L^{p'}(B_R)}^{p'},
\end{align*}
for any $\delta\in(0,1]$ with $c=c(n,p,s)$. 
The claimed comparison estimate follows in view of the Poincar\'e-type inequality from Lemma~\ref{lem:coz17b-4.7}; in fact, multiplying the last inequality by $c(1-s)(2R)^{sp}$ the left-hand side of the resulting inequality bounds the $L^p$ integral of $u-v$ in $B_{2R}(x_o)$, which is the same as the $L^p$-integral in  $B_R(x_o)$,
because $u-v\equiv 0$  in $\R^n\setminus B_R(x_o)$.
\end{proof}

By $\alpha_o\in(0,1)$ we continue to denote the constant defined in \eqref{def:alpha-o}. Now, we present the comparison estimate for the gradient.

\begin{proposition} 
[Comparison at the gradient level]
\label{lem:comparison-estimate}
Let $p >1$, $s\in(0,1)$ be such that $sp'>1$, $\alpha\in (0,\alpha_o)$, $f \in L^{p'}_{\loc}(\Omega)$, and $B_{2R}\equiv B_{2R}(x_o) \Subset \Omega$. Moreover,  let $u$ be a weak solution to~\eqref{eq:frac-p-lap} in the sense of Definition~\ref{def:weak-sol} and $v$ the unique weak solution to the homogeneous Dirichlet problem \eqref{CD-homo}. 
Then, for any $\delta\in(0,1]$ we have
\begin{align*}
    &\int_{B_{\frac12 R}} |\nabla u - \nabla v|^p \, \d x \\
    &\quad\,\leq 
    \delta \Big[ \|\nabla u\|_{L^p(B_{2R})}^p + 
    R^{n -p} \mathrm{Tail} \big(u-(u)_{B_{R}}; B_R\big)^p \Big] + \frac{c}{\delta^{\frac{1+\alpha}{\alpha(p-1)}}}
R^{p(sp'-1)} \|f\|_{L^{p'}(B_{R})}^{p'},
\end{align*}
where $c=c(n,p,s,\alpha)$. 
\end{proposition}

\begin{proof}
As already mentioned in Remark \ref{rem:W^1p}, by Theorem~\ref{thm:diening-nondifferentiabledata} we have $u\in W^{1+\alpha, p}_{\loc}(\Omega)$ and $v\in W^{1+\alpha, p}_{\loc}(B_R)$. 
This allows to apply first Corollary~\ref{cor:GN} and then Young's inequality with exponents $1+\alpha$ and $\frac{1+\alpha}{\alpha}$ to get that
\begin{align}\label{est:nabla(u-v)}\nonumber
    \| \nabla u &- \nabla v \|_{L^p(B_{\frac12 R})} \\\nonumber
    &\leq 
    \frac{c}{R} \| u-v \|_{L^p(B_{\frac12 R})} +
    c\, \| u-v \|_{L^p(B_{\frac12 R})}^{\frac{\alpha}{1+\alpha}}
    [\nabla u - \nabla v]_{W^{\alpha,p}(B_{\frac12 R})}^{\frac{1}{1+\alpha}} \\
    &\le 
    \frac{c}{R\epsilon^{\frac1\alpha}} 
    \| u-v \|_{L^p(B_{\frac12 R})} +
    \epsilon 
    \Big[R^{\alpha}[\nabla u]_{W^{\alpha,p}(B_{\frac12 R})} + R^{\alpha}[\nabla v]_{W^{\alpha,p}(B_{\frac12 R})}\Big] ,
\end{align}
holds true for any $\epsilon\in(0,1)$ with a constant $c=c(n,p,\alpha)$.
Next, we proceed to estimate the fractional semi-norms on the right-hand side.
To this end, from Theorem~\ref{thm:diening-nondifferentiabledata} and Poincar\'e's inequality we have
\begin{align*}
    &R^{\alpha} [\nabla u]_{W^{\alpha, p}(B_{\frac12 R})} \\
    &\qquad\leq 
      \frac{c}{R}\|u - (u)_{B_{R}}\|_{L^p(B_{R})} +  c\,R^{\frac{n}{p}-1} \mathrm{Tail} \big(u-(u)_{B_{R}};B_{R}\big)
    +
    c\, R^{sp'-1} \|f\|_{L^{p'}(B_{R})}^\frac{1}{p-1} \\
    &\qquad\le 
    c\,\|\nabla u\|_{L^p(B_{R})} + 
   c\, R^{\frac{n}{p}-1} \mathrm{Tail} \big(u-(u)_{B_{R}};B_{R}\big) +
    c\, R^{sp'-1} \|f\|_{L^{p'}(B_{R})}^\frac{1}{p-1} ,
\end{align*}
and also (in this case we apply Theorem~\ref{thm:diening-nondifferentiabledata} to the homogeneous equation) 
\begin{align*}
    R^{\alpha} [\nabla v]_{W^{\alpha, p}(B_{\frac12 R})} 
    &\leq 
    \frac{c}{R}\|v - (v)_{B_{R}}\|_{L^p(B_{R})} +
    c\,R^{\frac{n}{p}-1} \mathrm{Tail} \big(v-(v)_{B_{R}};B_{R}\big) .
\end{align*}
Note that $v \in W^{s,p}(B_R) \cap L^{p-1}_{\loc}(\R^n)$, which allows to apply Theorem~\ref{thm:diening-nondifferentiabledata} on $B_R$. Here again, the first term on the right-hand side of the last display is bounded by Poincar\'e's inequality  by 
\begin{align*}
    \|v - (v)_{B_{R}}\|_{L^p(B_{R})} 
    &\leq 
    2\|v-u\|_{L^p(B_{R})} +
    \|u - (u)_{B_{R}}\|_{L^p(B_{R})} \\
    &\leq 
    2\|v-u\|_{L^p(B_{R})} + 
    c\,R\|\nabla u\|_{L^p(B_{R})}, 
\end{align*}
whereas the second term, i.e.~the tail of $v-(v)_{B_{R}}$, can be bounded by 
\begin{align*}
    \mathrm{Tail} \big(v-(v)_{B_{R}};B_{R}\big) 
    &\leq 
    \mathrm{Tail} \big(v-(u)_{B_{R}};B_{R}\big) + 
    \mathrm{Tail} \big((u)_{B_R}-(v)_{B_{R}};B_{R}\big) \\
    &=: 
    \mathbf T_1 + \mathbf T_2.
\end{align*}
Let us treat $\mathbf T_1$ and $\mathbf T_2$ separately. For $\mathbf T_1$, we have
\begin{align*}
    \mathbf T_1 
    = 
    \bigg[ R^{sp} \int_{\R^n \setminus B_{R}} \frac{|v-(u)_{B_R}|^{p-1}}{|x-x_o|^{n+sp}} \, \d x \bigg]^\frac{1}{p-1} 
    =  
     \mathrm{Tail} \big(u- (u)_{B_R}; B_R\big),
\end{align*}
since $v=u$ on $\R^n\setminus B_R$.
For $\mathbf T_2$ we obtain
\begin{align*}
    \mathbf T_2 
    \leq 
    c\, \big| (u)_{B_R} - (v)_{B_{R}} \big| 
    &\leq 
    c\, R^{-\frac{n}{p}} \|u-v\|_{L^p(B_R)} .
\end{align*}
In conclusion, we have 
\begin{align*}
    &R^{\alpha} [\nabla v]_{W^{\alpha, p}(B_{\frac12 R})} \\
    &\qquad\leq 
    \frac{c}{R}  \|v-u\|_{L^p(B_{R})} +
    c \, \|\nabla u\|_{L^p(B_{R})} +
    c\,  R^{\frac{n}{p}-1} \mathrm{Tail} \big(u-(u)_{B_{R}};B_{R}\big) .
\end{align*}
Inserting the estimates for $[\nabla v]_{W^{\alpha, p}(B_{\frac12 R})}$ and $[\nabla u]_{W^{\alpha, p}(B_{\frac12 R})}$ into \eqref{est:nabla(u-v)} we
obtain
\begin{align*}
    \| \nabla u - \nabla v \|_{L^p(B_{\frac12 R})}^p 
    &\leq 
    \frac{c}{R^p\epsilon^{\frac{p}{\alpha}}}
    \| u-v \|_{L^p(B_{ R})}^p +
    c\,R^{p(sp'-1)} \|f\|_{L^{p'}(B_{R})}^{p'} \\
    &\quad +
    c\,\epsilon^p\Big[\|\nabla u\|_{L^p(B_{R})}^p + 
    R^{n-p} \mathrm{Tail} \big(u-(u)_{B_{R}};B_{R}\big)^p\Big] .
\end{align*}
To estimate the first term on the right-hand side, we apply the comparison estimate from Corollary~\ref{cor:u-v-fractest} with a parameter $\tilde\delta$ instead of $\delta$.  This provides us with
\begin{align*}
    \| \nabla u &- \nabla v \|_{L^p(B_{\frac12 R})}^p \\
    &\leq 
    c\bigg(
    \frac{\tilde\delta}{\epsilon^\frac{p}{\alpha}}+\epsilon^p
    \bigg)\|\nabla u\|_{L^p(B_{2R})}^p
    +
    c\,R^{p(sp'-1)}
    \bigg(1+\frac1{\tilde\delta^\frac1{p-1}}
    \bigg)\|f\|_{L^{p'}(B_{R})}^{p'}
    \\
    &\phantom{\le\,}+
    c\,\epsilon^p 
    R^{n-p} \mathrm{Tail} \big(u-(u)_{B_{R}};B_{R}\big)^p,
\end{align*}
where the parameters $\epsilon$ and $\tilde\delta$ are still at our disposal. Given $\delta\in (0,1]$ we first choose $\epsilon$ such that $c\,\epsilon^p=\tfrac12\delta$. This specifies $\epsilon$ in dependence on $\delta$, and consequently converts the term  $c\,\tilde\delta\,\epsilon^{-\frac{p}{\alpha}}$ into $c\,\tilde\delta
\big(\frac{\delta}{2c}\big)^{-\frac1\alpha}$. Next, we choose $\tilde\delta$ in such a way that
$c\, \tilde\delta= \tfrac12 \delta\big(\frac{\delta}{2c}\big)^\frac1\alpha$. With these choices we obtain
\begin{align*}
    \| \nabla u &- \nabla v \|_{L^p(B_{\frac12 R})}^p \\
    &\leq 
    \delta \Big[ 
    \|\nabla u\|_{L^p(B_{2R})}^p
    +
    R^{n-p} \mathrm{Tail} \big(u-(u)_{B_{R}};B_{R}\big)^p\Big]
    +
    \frac{c\, R^{p(sp'-1)}}{\delta^\frac{1+\alpha}{\alpha(p-1)}}\|f\|_{L^{p'}(B_{R})}^{p'},
\end{align*}
where $c=c(n,p,s,\alpha)$. This proves the claim. 
\end{proof}

\section{Proof of Theorem \ref{thm:main}}\label{S:proof}

We first prove the result in the special case $R=1$ and $x_o=0$ and then conclude the general case by a re-scaling argument via Lemma~\ref{lem:equation-scaling}.  The setting is as follows. Let $1\le s_1<s_2\le 2$. Consider balls $B_1 \subset B_{s_1} \subset B_{s_2} \subset B_2$  centered at the origin,
and define
\begin{align}\label{def:lambda_0}
    \lambda_o^p 
    := 
    \bint_{B_2} \big[|\nabla u|^p + \mathsf M^p |f|^{p'} \big]\, \d x + 
    \mathrm{Tail} \big(u-(u)_{B_2};B_2\big)^p ,
\end{align}
for some $\mathsf M\ge 1$ to be chosen later.

\subsection{Stopping time argument.}
For any point $y_o \in B_{s_1}$ and radius $r$ with
\begin{equation*}
    \frac{s_2-s_1}{2^7} \leq r \leq \frac{s_2-s_1}{2},
\end{equation*}
we have $B_r(y_o)\subset B_2$.
We define 
\begin{equation}\label{def:B}
    \mathfrak B:=
    \bigg(
    \frac{2^{8}}{s_2-s_1}
    \bigg)^{\frac{n}{p-1}+1}.
\end{equation}
Note that $\frac{n}{p-1}+1>\frac{n}{p}$. 
Then
\begin{equation} \label{eq:lambda-sub}
    \bint_{B_r(y_o)} 
    \big[|\nabla u|^p +\mathsf  M^p |f|^{p'} \big]\, \d x < \lambda^p
\end{equation}
for any $r$ in the above range and any
\begin{equation} \label{eq:lambda-geq-Blambda0}
    \lambda > \mathfrak B \lambda_o.
\end{equation}
Now let $\lambda$ be as in \eqref{eq:lambda-geq-Blambda0} and consider the {\it super-level set}
$$
    \boldsymbol E(\lambda,s_1) 
    := 
    \big\{\mbox{$x \in B_{s_1}$: $x$ is a Lebesgue point of $ |\nabla u|$ and $|\nabla u|(x) > \lambda$}\big\}.
$$
Lebesgue's differentiation theorem ensures that for every $y_o\in \boldsymbol E(\lambda,s_1)$ it holds that 
\begin{align*}
    \lim_{\rho\downarrow 0} \bint_{B_\rho(y_o)} 
    \big[
    |\nabla u|^p +\mathsf M^p |f|^{p'}\big] \, \d x 
    \ge
    |\nabla u(y_o)|^p
    >
    \lambda^p.
\end{align*}
Taking \eqref{eq:lambda-sub} into account, the absolute continuity of the integral ensures the existence of a maximal radius $\rho_{y_o}\in(0,\frac{s_2-s_1}{2^7})$, so that
\begin{equation*}
    \bint_{B_{\rho_{y_o}}(y_o)} 
    \big[
    |\nabla u|^p + \mathsf M^p |f|^{p'}\big] \, \d x 
    =
    \lambda^p 
\end{equation*}
and 
\begin{equation}\label{eq:lambda-sub-i}
    \bint_{B_{\rho}(y_o)} 
    \big[
    |\nabla u|^p +\mathsf M^p |f|^{p'}
    \big]\, \d x 
    <
    \lambda^p ,
    \qquad\mbox{for any $\rho\in (\rho_{y_o}, \frac{s_2-s_1}{2}]$.}
\end{equation}

\subsection{Covering of super-level sets}
Now consider the collection of balls \begin{equation*}
    \mathcal F:=\big\{B_{\rho_{y_o}}(y_o)\big\}_{y_o\in 
    \boldsymbol E(\lambda,s_1)}
\end{equation*}
constructed in the last section. By Vitali's covering theorem there exists a countable family of disjoint balls $\{B_i\}_{i \in \N} = \{B_{\rho_{x_i}}(x_i)\}_{i \in \N }\subset\mathcal F$ such that 
$$
    \boldsymbol E(\lambda,s_1)
    \subset 
    \bigcup_{B\in\mathcal F} B
    \subset 
    \bigcup_{i\in\N} B_{5\rho_{x_i}}(x_i).
$$
In particular, we have  for any $i\in\N$ that $x_i \in \boldsymbol E(\lambda,s_1)$, $\rho_{x_i}\in (0,\frac{s_2-s_1}{2^7})$, 
\begin{equation}\label{eq:intr}
    \bint_{B_i}\big[ |\nabla u|^p +\mathsf M^p |f|^{p'}\big] \, \d x = \lambda^p,
\end{equation}
and
\begin{equation} \label{eq:stopping-time-above-maxradius}
    \bint_{B_{\rho}(x_i)}
    \big[|\nabla u|^p + \mathsf M^p |f|^{p'} 
    \big]\, \d x 
    < 
    \lambda^p 
    \quad \mbox{ for all $\rho\in \big(\rho_{x_i}, \frac{s_2-s_1}{2}\big]$.}
\end{equation}
Finally, we introduce the abbreviations 
$$
\rho_i := \rho_{x_i},\quad B_i^{(j)} := B_{2^j \rho_{i}}(x_i),\quad \mbox{for $ i,\,j \in \N$.}
$$ 
 Clearly,  for each $i\in\N$ we have $B_i^{(7)}\subset B_2$.

\subsection{Comparison}
Let $v_i\in W^{s,p}(B_i^{(5)}) \cap L^{p-1}_{sp} (\R^n)$ be the unique weak solution to the Dirichlet problem
\begin{equation*} 
\left\{
\begin{array}{cl}
      (- \Delta_p)^s v_i  = 0 
      & 
      \mbox{in $ B_i^{(5)}$,} \\[7pt]
      v_i = u & \mbox{a.e.~in $\R^n \setminus B_i^{(5)}$.}
   \end{array}
\right.
\end{equation*}
Let $\mathsf A > 1$ and $\theta>p$ be constants to be determined later depending on the data and $\alpha=\frac12\alpha_o$, where $\alpha_o\in(0,1)$ is the constant defined in \eqref{def:alpha-o}. Applying Lemma~\ref{lem:elementary-superlevel} with $\big(\gamma,\alpha, K, a,b\big)$ replaced by $\big( p,\theta, \mathsf A\lambda, \nabla u(x),\nabla v_i(x)\big)$ and integrating over $B_i^{(3)}\cap \left\{ |\nabla u| > \mathsf A \lambda \right\}$,
we obtain
\begin{align} \label{eq:grad-u-superlevel}\nonumber
    &\int_{B_i^{(3)}\cap \left\{ |\nabla u| >\mathsf A \lambda \right\}} |\nabla u|^p \, \d x \\
    &\qquad\leq 
    2^{p} \int_{B_i^{(4)}} |\nabla u - \nabla v_i|^p \, \d x + 
    \frac{2^{\theta-1}}{(\mathsf A\lambda)^{\theta-p}} \int_{B_i^{(3)}} |\nabla v_i|^\theta \, \d x  
    =:  
    \mathbf{I} + \mathbf{II},
\end{align}
with the obvious meaning of $\mathbf{I}$
and $\mathbf{II}$. Note that in $\mathbf{I}$ we increased the domain of integration from $B_i^{(3)}$ to $B_i^{(4)}$. We treat  $\mathbf{I}$
and $\mathbf{II}$ separately.
For any $\delta\in(0,1]$, Proposition~\ref{lem:comparison-estimate} implies
\begin{align*}
    \mathbf{I} 
    &\leq 
    2^p \delta
    \bigg[ 
    \int_{B_i^{(6)}}|\nabla u|^p\,\d x
    + 
    \big(2^5\rho_i\big)^{n -p} \mathrm{Tail} 
    \big(u-(u)_{B_i^{(5)}};B_i^{(5)}
    \big)^p \bigg]\\
    &\phantom{\le\,}
    + 
    \frac{c}{\delta^\frac{1+\alpha}{\alpha(p-1)}} \big(2^5\rho_i\big)^{p(sp'-1)} 
    \int_{B_i^{(5)}}
    |f|^{p'}\,\d x .
\end{align*}
The three terms on the right-hand side are estimated as follows. First,
by~\eqref{eq:lambda-sub-i} we have  
\begin{equation*}
    \int_{B_i^{(6)}} |\nabla u|^p 
    \, \d x
    < \big|B_i^{(6)}\big| \lambda^p
    =
    2^{6n} |B_i|\lambda^p,
\end{equation*}
whereas the last term is bounded by
\begin{equation*}
    \big(2^5\rho_i\big)^{p(sp'-1)}
    \int_{B_i^{(5)}} |f|^{p'} \, \d x
    <
    \frac{\big|B_i^{(5)}\big|}{\mathsf M^p}\lambda^p
    <\frac{2^{5n}}{\mathsf M^p}|B_i|\lambda^p
\end{equation*}
since $2^5\rho_i \leq \frac12$. As for the tail-term, observe that due to $2^5\rho_i\in (0,\frac{s_2-s_1}{2})$ there exists $k\in\N_0$ such that \begin{equation*}
    \frac{s_2-s_1}{4}\le 2^k2^5\rho_i<\frac{s_2-s_1}{2}.
\end{equation*}
This allows us to apply Lemma~\ref{lem:tail-rho} on $B_i^{(5)}$ with $R:=\frac{s_2-s_1}{2}$ and obtain that (recall that the radius of $B_i^{(5)}$ is $2^5\rho_i$)
\begin{align*} 
    \mathrm{Tail} \big(u-(u)_{B_i^{(5)}}; B_i^{(5)}\big) 
    &\leq 
    \frac{c}{R^{\frac{n}{p-1}}}\Big(\frac{\rho_i}{R}\Big)^{sp'} 
    \Bigg[\mathrm{Tail} \big(u-(u)_{B_{2}};B_{2}\big) +
    \bigg[  \,
    \bint_{B_{2}}
    |\nabla u|^{p} \,\d x\bigg]^\frac{1}{p} \Bigg]\\
    &\quad + 
    c\, \rho_i \sum_{j=1}^{k} 2^{-j(sp' - 1)} \bigg[ \,\bint_{B_i^{(j+5)}} |\nabla u|^p \, \d x \bigg]^\frac{1}{p}  .
\end{align*}
Using  the definition of $\lambda_o$ from \eqref{def:lambda_0}, \eqref{eq:stopping-time-above-maxradius} and the assumption $sp'>1$, the right-hand side can be further estimated. Indeed, we have
\begin{align}\label{est:tail-B5} \nonumber
    \mathrm{Tail} \big(u-(u)_{B_i^{(5)}}; B_i^{(5)}\big) 
    &\leq 
    \frac{c}{R^{\frac{n}{p-1}}}\Big(\frac{\rho_i}{R}\Big)^{sp'} 
    \lambda_o + 
    c\, \rho_i \lambda\sum_{j=1}^{k} 2^{-j(sp' - 1)} \\\nonumber
    &\le
    \frac{c\,\rho_i\lambda }{R^{\frac{n}{p-1}+1}\mathfrak B} +
     c \,\rho_i \lambda\sum_{j=1}^{\infty} 2^{-j(sp' - 1)}\\
    &\le 
    c \,\rho_i\lambda.
\end{align}
To obtain the last line we used the definition of $\mathfrak B$ in \eqref{def:B} and $R=\frac{s_2-s_1}{2}$.
Inserting the preceding inequalities on the right-hand side of the above estimate for $\mathbf I$, we obtain
\begin{align}\label{est:I}
    \mathbf{I} 
    \leq 
    c \bigg[\delta + 
    \frac{1}{\delta^\frac{1+\alpha}{\alpha(p-1)} \mathsf M^p}\bigg]  |B_i| \lambda^p.
\end{align}
Note that \eqref{est:I} holds for any $\delta\in(0,1]$.

To estimate the term $\mathbf{II}$ in~\eqref{eq:grad-u-superlevel}, we apply Theorem~\ref{thm:Lq-gradient} for $v_i - (v_i)_{B_i^{(4)}}$ on $B_i^{(4)}$. Note that $sp'>1$ implies that $s>\frac{(p-2)_+}{p}$. 
Hence, we obtain 
\begin{align*}
    \mathbf{II} 
    &\leq 
    \frac{c}{(\mathsf A\lambda)^{\theta-p}} |B_i| 
    \left[ \bint_{B_i^{(4)}} |\nabla v_i|^p \,\d x + 
    \rho_i^{-p}\mathrm{Tail} \big(v_i-(v_i)_{B_i^{(4)}}; B_i^{(4)}\big)^p \right]^\frac{\theta}{p} \\
    &\leq 
    \frac{c}{(\mathsf A\lambda)^{\theta-p}} |B_i| \left[ 
     \frac{\mathbf I}{|B_i|}+
    \bint_{B_i^{(4)}} |\nabla u|^p \,\d x +
    \rho_i^{-p}\mathrm{Tail} \big(v_i-(v_i)_{B_i^{(4)}}; B_i^{(4)}\big)^p \right]^\frac{\theta}{p}, 
\end{align*}
where $c = c(n,p,s,\theta)> 0$.  The first term on the right-hand side can be  bounded via inequality~\eqref{est:I} with the choice $\delta=1$, while the  second one can be bounded by~\eqref{eq:stopping-time-above-maxradius}. Finally, for the third term we have 
\begin{align*}
    \mathrm{Tail} \big(v_i-(v_i)_{B_i^{(4)}}; B_i^{(4)}\big)
    &\leq 
     c\big[\mathbf T_1 + \mathbf T_2 + \mathbf T_3\big],
\end{align*}
where
\begin{equation*}
    \mathbf T_1
    :=
    \mathrm{Tail} \big(u-(u)_{B_i^{(4)}}; B_i^{(4)}\big),
    \qquad
    \mathbf T_2
    :=
    \mathrm{Tail} \big(v_i-u;B_i^{(4)}\big),
\end{equation*}
and
\begin{equation*}
    \mathbf T_3
    :=
    \mathrm{Tail} \big((u)_{B_i^{(4)}}-(v_i)_{B_i^{(4)}};B_i^{(4)}\big).
\end{equation*}
 We aim to the term $\mathbf T_1$ by \eqref{est:tail-B5}. 
This is achieved by first splitting the domain of integration into $\R^n\setminus  B_i^{(5)}$ 
and $B_i^{(5)}\setminus B_i^{(4)}$. The local integral is estimated by using  $|x-x_o|\ge 2^4\rho_i$ and then H\"older's inequality to raise the power from $p-1$ to $p$. Next, in both the local integral and the tail term
we add and subtract the mean value of $u$ in $B_i^{(5)}$ which creates a third term--the difference of mean values. However, it can be readily estimated via H\"older's inequality. Then, Poincar\'e's inequality can be applied to estimate the local integral. Specifically, we have
\begin{align*}
    \mathbf T_1
    &\le
    c \Tail \big(u-(u)_{B_i^{(4)}}; B_i^{(5)}\big)
    +
    c
    \Bigg[
    (2^4\rho_i)^{sp}
    \int_{B_i^{(5)}
    \setminus B_i^{(4)}}
    \frac{|u-(u)_{B_i^{(4)}}|^{p-1}}{|x-x_o|^{n+sp}}\,\dx
    \Bigg]^\frac1{p-1}\\
    &\le
   c \Tail \big(u-(u)_{B_i^{(4)}}; B_i^{(5)}\big)
   +c\bigg[
    \bint_{B_i^{(5)}}
    |u-(u)_{B_i^{(4)}}|^{p-1}\,\dx
    \bigg]^\frac1{p-1}\\
    &\le
    c \Tail \big(u-(u)_{B_i^{(4)}}; B_i^{(5)}\big)
   +c\bigg[
    \bint_{B_i^{(5)}}
    |u-(u)_{B_i^{(4)}}|^{p}\,\dx
    \bigg]^\frac1{p}\\
    &\le
    c \Tail \big(u-(u)_{B_i^{(5)}}; B_i^{(5)}\big)
    +
    c\bigg[
    \bint_{B_i^{(5)}}
    |u-(u)_{B_i^{(5)}}|^{p}\,\dx
    \bigg]^\frac1{p}
    +c\,\big|(u)_{B_i^{(5)}}-(u)_{B_i^{(4)}}\big|
    \\
    &\le
    c \Tail \big(u-(u)_{B_i^{(5)}}; B_i^{(5)}\big)
    +
    c\bigg[
    \bint_{B_i^{(5)}}
    |u-(u)_{B_i^{(5)}}|^{p}\,\dx
    \bigg]^\frac1{p}\\
    &\le
    \underbrace{c\Tail \big(u-(u)_{B_i^{(5)}}; B_i^{(5)}\big)}_{\le c\,\rho_i\lambda,\;\mbox{\footnotesize by \eqref{est:tail-B5}}}
    +\,
    c\,\rho_i\underbrace{\bigg[\bint_{B_i^{(5)}}
    |\nabla u|^{p}\,\dx
    \bigg]^\frac1{p} }_{\le\lambda,\;\mbox{\footnotesize by \eqref{eq:stopping-time-above-maxradius}}}\le c\,\rho_i\lambda.
\end{align*}
For $\mathbf T_2$ we recall that $v_i=u$ on $\R^n \setminus B_i^{(5)}$. Subsequently, applying Hölder's inequality and  Corollary~\ref{cor:u-v-fractest} we obtain
\begin{align*}
    \mathbf T_2 +\mathbf T_3 
    &= 
    \bigg[(2^4\rho_i)^{sp} \int_{B_i^{(5)} \setminus B_i^{(4)}} \frac{|v_i-u|^{p-1}}{|x-x_i|^{n+sp}} \, \d x \bigg]^{\frac{1}{p-1}} +
    c\,\big|(u)_{B_i^{(4)}}-(v_i)_{B_i^{(4)}}\big|\\
    &\le 
    c \bigg[\bint_{B_i^{(5)}} |v_i-u|^{p-1} \,\d x\bigg]^{\frac{1}{p-1}} +
    c\,\bint_{B_i^{(5)}} |v_i-u| \,\d x\\
    &\le 
    c \bigg[\bint_{B_i^{(5)}} |v_i-u|^{p} \,\d x\bigg]^{\frac{1}{p}} \\
    &\le 
    c\bigg[\rho_i^p\bint_{B_i^{(6)}} |\nabla u|^p \, \d x +
    \rho_i^{spp'} \bint_{B_i^{(5)}}  |f|^{p'} \, \d x \bigg]^{\frac{1}{p}} \\
    &\le 
    c\bigg[
    \rho_i^p\lambda^p
    +\rho_i^{spp'}\frac{\lambda^p}{\mathsf M^p}
    \bigg]\\
    &\le 
    c\,\rho_i\lambda ,
\end{align*}
where to obtain the second last line we used \eqref{eq:stopping-time-above-maxradius}. Moreover, we took advantage of $\rho_i\le 1$ and $\mathsf M\ge 1$ in the passage from the penultimate to the last line.  Overall, we have shown that 
$$
    \mathbf{II} 
    \leq 
    \frac{c}{\mathsf A^{\theta-p}} \lambda^p |B_i|,
$$
for a constant $c = c(n,p,s,\theta)> 0$. Together with the estimate for $\mathbf{I}$ we conclude that
\begin{align}\label{est:super-nabla-u}
    \int_{B_i^{(3)}\cap \left\{ |\nabla u| >\mathsf A \lambda \right\}} |\nabla u|^p \, \d x 
    \le 
    c\, \boldsymbol{\mathfrak{M}} \, \lambda^p |B_i| ,
\end{align}
where 
\begin{equation}\label{def:M}
    \boldsymbol{\mathfrak{M}}
    :=
    \delta + 
    \frac{1}{\delta^\frac{1+\alpha}{\alpha(p-1)} \mathsf M^p} +
    \frac{1}{\mathsf A^{\theta-p}}.
\end{equation}

\subsection{Estimates on super-level sets}
In \eqref{eq:intr} we split the domain of integration for $u$ into the super-level set $B_i \cap \{|\nabla u| > \lambda/4 \}$ and the sub-level set $B_i \cap \{|\nabla u| \le \lambda/4 \}$, and for $f$ into $B_i \cap \{|f| > (\lambda/(4\mathsf M))^{p-1} \}$ and its complement.
On the respective sub-level sets we estimate $u$ and $f$ by their upper bounds.
This yields
\begin{align*}
    |B_i|\lambda^p
    &=
    \int_{B_i}\big[ |\nabla u|^p +\mathsf M^p |f|^{p'}\big] \, \d x \\
    &= 
    \int_{B_i \cap \{|\nabla u| > \lambda/4 \}} |\nabla u|^p \, \d x + 
    \mathsf M^p \int_{B_i \cap \{|f| > (\lambda/(4\mathsf M))^{p-1} \}} |f|^{p'} \, \d x \\
    &\quad +
    \int_{B_i \cap \{|\nabla u| \le \lambda/4 \}} |\nabla u|^p \, \d x + 
    \mathsf M^p \int_{B_i \cap \{|f| \le (\lambda/(4\mathsf M))^{p-1} \}} |f|^{p'} \, \d x \\
    &\le 
    \int_{B_i \cap \{|\nabla u| > \lambda/4 \}} |\nabla u|^p \, \d x + 
    \mathsf M^p \int_{B_i \cap \{|f| > (\lambda/(4\mathsf M))^{p-1} \}} |f|^{p'} \, \d x +
    \tfrac{2}{4^p}  |B_i| \lambda^p.
\end{align*}
Re-absorb $\frac{2}{4^p}  |B_i|\lambda^p$ into the left-hand side and  insert the result into \eqref{est:super-nabla-u}. In this way, we obtain
\begin{align*}
    &\int_{B_i^{(3)}\cap \left\{ |\nabla u| > \mathsf A \lambda \right\}} |\nabla u|^p \, \d x \\
    &\quad\leq 
    c\,\boldsymbol{\mathfrak{M} }
    \bigg[\int_{B_i \cap \{|\nabla u| > \lambda/4 \}} |\nabla u|^p \, \d x + 
    \mathsf M^p \int_{B_i \cap \{|f| > (\lambda/(4\mathsf M))^{p-1} \}} |f|^{p'} \, \d x \bigg].
\end{align*}
Recall that $\{B_i^{(3)} \}_{i\in \N}$ covers the super-level set $\boldsymbol E(\lambda,s_1)$, and thus also $\boldsymbol E(\mathsf A \lambda,s_1)$, and that the balls $B_i$ are pairwise disjoint and contained in $B_{s_2}$. Consequently,
\begin{align*}
    &\int_{\boldsymbol E(\mathsf A \lambda,s_1)} |\nabla u|^p \, \d x \\
    &\qquad\le
    \sum_{i\in\N}
    \int_{B_i^{(3)}\cap \left\{ |\nabla u| >\mathsf A \lambda \right\}} |\nabla u|^p \, \d x\\
    &\qquad\le
    c\,\boldsymbol{\mathfrak{M} }
    \sum_{i\in\N}
    \bigg[\int_{B_i \cap \{|\nabla u| > \lambda/4 \}} |\nabla u|^p \, \d x + 
    \mathsf M^p \int_{B_i \cap \{|f| > (\lambda/(4\mathsf M))^{p-1} \}} |f|^{p'} \, \d x \bigg]\\
    &\qquad\leq 
    c\, \boldsymbol{\mathfrak{M} } \bigg[\int_{\boldsymbol E(\lambda/4,s_2)} |\nabla u|^p \, \d x + \mathsf M^p \int_{B_{s_2} \cap \{|f| > (\lambda/(4\mathsf M))^{p-1} \}} |f|^{p'} \, \d x \bigg].
\end{align*}
Instead of $\boldsymbol{E}(\lambda,s_1)$ we consider for $k\ge\mathfrak B\lambda_o$ the set
\begin{equation*}
    \boldsymbol{E}_k(\lambda,s_1) 
    := 
    \Big\{x \in B_{s_1} :x \text{ is a Lebesgue point of } |\nabla u| \text{ and } |\nabla u|_k(x) > \lambda\Big\}
\end{equation*}
where $|\nabla u|_k := \min \big\{ |\nabla u|, k \big\}$ is the truncation of $|\nabla u|$ at height $k$. We rewrite the last inequality in terms of $\boldsymbol E_k$ and obtain
\begin{align*}
    &\int_{\boldsymbol E_k(\mathsf A \lambda,s_1)} |\nabla u|^p \, \d x \\
    &\qquad\leq 
    c\, \boldsymbol{\mathfrak{M}}
    \bigg[\int_{\boldsymbol E_k(\lambda/4,s_2)} |\nabla u|^p \, \d x + \mathsf M^p \int_{B_{s_2} \cap \{|f| > (\lambda/(4\mathsf M))^{p-1} \}} |f|^{p'} \, \d x \bigg].
\end{align*}
We multiply the inequality by $\lambda^{pq-p-1}$ and integrate with respect to $\lambda$ over the interval $(\mathfrak B \lambda_o, \infty)$. For the integral on the left-hand side, using a Fubini type argument, we get
\begin{align*}
    &\int_{\mathfrak B \lambda_o}^\infty \lambda^{pq-p-1} \int_{\boldsymbol E_k(\mathsf A\lambda,s_1)} |\nabla u|^p \, \d x \d \lambda\\
    &\qquad = 
    \frac{1}{pq-p} \bigg[ \mathsf A^{p-pq} \int_{B_{s_1}} |\nabla u|_k^{pq-p} |\nabla u|^p  \, \d x  - (\mathfrak B \lambda_o)^{pq-p} \int_{B_{s_1}} |\nabla u|^p \, \d x \bigg].
\end{align*}
The terms on the right-hand side can be estimated similarly as
\begin{align*}
    \int_{\mathfrak B \lambda_o}^\infty \lambda^{pq-p-1} \int_{\boldsymbol E_k(\lambda/4,s_2)} |\nabla u|^p \, \d x \d \lambda
    &\leq 
    \frac{4^{pq-p}}{pq-p} \int_{B_{s_2}} |\nabla u|_k^{pq-p} |\nabla u|^p  \, \d x,
\end{align*}
and
\begin{align*}
    \int_{\mathfrak B \lambda_o}^\infty \lambda^{pq-p-1} \int_{B_{s_2} \cap \{|f| > (\lambda/(4\mathsf M))^{p-1} \}} |f|^{p'} \, \d x \d \lambda  \leq \frac{(4\mathsf M)^{pq-p}}{pq-p} \int_{B_{s_2}} |f|^{qp'} \, \d x.
\end{align*}
Combining these estimates results in
\begin{align*}
    \int_{B_{s_1}} &|\nabla u|_k^{pq-p} |\nabla u|^p  \, \d x \\
    &\leq 
    c_\ast \mathsf A^{pq-p} \boldsymbol{\mathfrak{M}} \int_{B_{s_2}} |\nabla u|_k^{pq-p} |\nabla u|^p  \, \d x\\
    &\phantom{\le\,} +
     (\mathsf A\mathfrak B\lambda_o)^{pq-p} \int_{B_{s_1}} |\nabla u|^p \, \d x  + 
    c_\ast \mathsf A^{pq-p} \boldsymbol{\mathfrak{M}} 
   \mathsf M^{pq} \int_{B_{s_2}} |f|^{qp'} \, \d x,
\end{align*}
where $\boldsymbol{\mathfrak M}$ is defined in \eqref{def:M}. The strategy is now to absorb the first term on the right-hand side into the left-hand side. This requires careful choices of the parameters in $\boldsymbol{\mathfrak{M}}$, i.e.~of $\theta$, $\mathsf A$, $\delta$, and $\mathsf M$.
We start by choosing $\theta = 2pq$. Next, we choose $\mathsf A$ large enough such that
\begin{equation*}
     \frac{c_\ast}{\mathsf A^{pq}} 
    = 
    \tfrac 16\quad\Longleftrightarrow\quad
    \mathsf A=
     (6c_\ast)^\frac1{pq}.
\end{equation*}
Subsequently, we choose $\delta$ 
small enough such that 
\begin{equation*}
    c_\ast \mathsf A^{pq-p} \delta \equiv 
    c_\ast(6 c_\ast)^{\frac{q-1}{q}}\delta
    = \tfrac16\quad\Longleftrightarrow\quad
    \delta= (6c_\ast)^{\frac1q-2}.
\end{equation*}
Finally we choose $\mathsf M$ large enough such that 
\begin{align*}
    \frac{c_\ast \mathsf A^{pq-p}}{\delta^\frac{1+\alpha}{\alpha (p-1)} \mathsf M^p} 
    \equiv
    \frac{c_\ast (6c_\ast)^{\frac{q-1}{q} + (2-\frac1q)\frac{1+\alpha}{\alpha(p-1)}}}{\mathsf M^p}
    = 
    \tfrac16
    \quad\Longleftrightarrow\quad
    \mathsf M^p=(6c_\ast)^{(2-\frac1q)\frac{1+\alpha p}{\alpha(p-1)}}.
\end{align*}
These choices ensure that, on the one hand, $c_\ast  \mathsf A^{pq-p} \boldsymbol{\mathfrak{M}}=\frac12$, and on the other hand $\mathsf M$ depends only on $n,p,q,s$ (recall that $\alpha=\alpha(p,s)$). 
With the aforementioned choices, we obtain
\begin{align*}
    \int_{B_{s_1}} &|\nabla u|_k^{pq-p} |\nabla u|^p  \, \d x \\
    &\leq 
   \tfrac12 \int_{B_{s_2}} |\nabla u|_k^{pq-p} |\nabla u|^p  \, \d x \\
   &\quad +  
   (\mathsf A\mathfrak B\lambda_o)^{pq-p} \int_{B_{2}} |\nabla u|^p \, \d x  + 
    \tfrac12 \mathsf M^{pq} \int_{B_{s_2}} |f|^{qp'} \, \d x \\
    &\leq 
    \tfrac12 \int_{B_{s_2}} |\nabla u|_k^{pq-p} |\nabla u|^p \, \d x \\
    &\quad + 
    \frac{c\, \lambda_o^{pq-p}}{(s_2-s_1)^{\beta }} \int_{B_2} |\nabla u|^p \, \d x  + 
    c \int_{B_2} |f|^{qp'} \, \d x,
\end{align*}
where $\beta = (np'+p)(q-1)$. To obtain the last line we also used the definition of $\mathfrak B$ in~\eqref{def:B}. Note that the last inequality holds for any $1\le s_1<s_2\le 2$. Therefore, the iteration lemma~\ref{lem:tech} can be applied. The application results in
\begin{equation*}
     \int_{B_{1}} |\nabla u|_k^{pq-p} |\nabla u|^p  \, \d x
     \le
     c\,
     \lambda_o^{pq-p} \int_{B_2} |\nabla u|^p \, \d x + c \int_{B_2} |f|^{qp'} \, \d x.
\end{equation*}
Using Fatou's lemma we pass to the limit  $k \to \infty$ and obtain
\begin{align*}
    \int_{B_1} |\nabla u|^{pq}  \, \d x  
    &\leq 
    c\, \lambda_o^{pq-p} \int_{B_2} |\nabla u|^p \, \d x + 
    c \int_{B_2} |f|^{qp'} \, \d x.
\end{align*}
Recalling the definition of $\lambda_o$ from \eqref{def:lambda_0} and passing  to the mean values, we get
\begin{align*}
    \bigg[ \bint_{B_{1}} |\nabla u|^{pq} \, \d x \bigg]^\frac{1}{q}
    \leq 
    c \,\bint_{B_{2}}  |\nabla u|^p  \, \d x + c \bigg[ \bint_{B_2} |f|^{qp'} \, \d x \bigg]^\frac{1}{q} + c \, \mathrm{Tail} \big(u-(u)_{B_2};B_2\big)^{p}.
\end{align*}
Scaling back to $B_{2R}$ with the aid of Lemma~\ref{lem:equation-scaling} finally yields, 
\begin{align*}
\bigg[ \bint_{B_{R}} |\nabla u|^{pq} \, \d x \bigg]^\frac{1}{q} &\leq c\, \bint_{B_{2R}} |\nabla u|^p \, \d x + \frac{c}{R^{(1-s)p}} \bigg[ \bint_{B_{2R}} |R^s f|^{qp'} \, \d x \bigg]^\frac{1}{q} \\
&\quad  + \frac{c}{R^p} \mathrm{Tail} \big(u-(u)_{B_{2R}};B_{2R}\big)^p.
\end{align*}
This concludes the proof.

\medskip
\noindent
{\bf Acknowledgments.}  N.~Liao is supported by the Austrian Science Fund (FWF) project \emph{On the Stefan type problems}, 
Grant DOI 10.55776/P36272.
K.~Moring is supported by the the Austrian Science Fund (FWF)
project \emph{Widely degenerate partial differential equations}, 
Grant DOI 10.55776/P36295.
Our thanks also go to the Erwin Schr\"odinger International Institute for Mathematics and Physics (Vienna, Austria), where the final version of this work was completed during the workshop \emph{Degenerate and Singular PDEs}.

\end{document}